\documentclass[11pt]{amsart}

\usepackage{amsmath,amsfonts}
\usepackage{amssymb}
\usepackage{amscd}
\usepackage{amsthm}
\usepackage{yhmath}
\usepackage{subfigure}
\usepackage{graphicx}   

\usepackage[all]{xy}
\usepackage{color}

\usepackage{marginnote}

\numberwithin{equation}{section}

\newtheorem{theorem}[equation]{Theorem}

\newtheorem{proposition}[equation]{Proposition}

\newtheorem{lem}[equation]{Lemma}
\newtheorem{lemma}[equation]{Lemma}

\newtheorem{characterization}[equation]{Characterization}

\theoremstyle{remark}

\theoremstyle{definition}
\newtheorem{definition}[equation]{Definition}

\newcommand{\abs}[1]{\lvert#1\rvert}

\def\XXint#1#2#3{{\setbox0=\hbox{$#1{#2#3}{\int}$}
	\vcenter{\hbox{$#2#3$}}\kern-.5\wd0}}

\newcommand{\N}{\mathbb N}

\newcommand{\R}{\mathbb R}

\newcommand{\TT}{\mathcal T}
\newcommand{\DD}{\mathcal D}
\newcommand{\RR}{\mathcal R}
\newcommand{\PP}{\mathcal P}

\newcommand{\HH}{\mathbb H}

\newcommand{\im}{\operatorname{Im}}

\newcommand{\card}{\operatorname{Card}}
\newcommand{\cone}{\operatorname{Cone}}

\newcommand{\al}{\alpha}

\def\eps{\epsilon}

\def\th{\theta}

\begin{document}

\title[Besicovitch Covering Property in Heisenberg groups]{Besicovitch Covering Property  for homogeneous distances on the Heisenberg groups}

\author{Enrico Le Donne}

\address[Le Donne]{Department of Mathematics and Statistics, P.O. Box 35,
FI-40014,
University of Jyv\"askyl\"a, Finland}
\email{ledonne@msri.org}

\author{S\'everine Rigot}
\address[Rigot]{Laboratoire de Math\'ematiques J.A. Dieudonn\'e UMR CNRS 7351,  Universit\'e Nice Sophia Antipolis, 06108 Nice Cedex 02, France}
\email{rigot@unice.fr}

\thanks{The work of S.R. is supported by the ANR-12-BS01-0014-01 Geometrya.}

\renewcommand{\subjclassname} {\textup{2010} Mathematics Subject Classification}
\subjclass[]{ 
28C15, %  	Set functions and measures on topological spaces (regularity of measures, etc.)
49Q15, %  Geometric measure and integration theory, integral and normal currents
43A80. % (1973-now) Analysis on other specific Lie groups
%53C17, %   Sub-Riemannian geometry
%53C60,   % Finsler spaces and generalizations 
%28A75,  %  Length, area, volume, other geometric measure theory
%26A16  % Lipschitz (Hlder) classes
%58C35   Integration on manifolds; measures on manifolds
%26B20 Integral formulas (Stokes, Gauss, Green, etc.)
%54Exx, % Spaces with richer structures 
%37L40 %Invariant measures
%58D05, %Groups of diffeomorphisms and homeomorphisms as manifolds
%22F50, %Groups as automorphisms of other structures
% 22DXX % Locally compact groups and their algebras
%22E25, % Nilpotent and solvable Lie groups
% 22F30 % Homogeneous spaces
%14M17. %Homogeneous spaces and generalizations 
% 53C30 % Homogeneous manifolds
% 58D19% Group actions and symmetry properties
% 58C25 % Differentiable maps
%49J20  % (1991-now) Optimal control problems involving partial differential equations
%49K21,  % (2010-now) Problems involving relations other than differential equations
%49J15  %(1991-now) Optimal control problems involving ordinary differential equations
}

\keywords{Covering theorems, Heisenberg groups, Homogeneous distances}

\date{July 4, 2014}

\begin{abstract} 
Our main result is a positive answer to the question whether one can find  homogeneous distances on the Heisenberg groups that have the Besicovitch Covering Property (BCP). This property is well known to be one of the fundamental tools of measure theory, with strong connections with the theory of differentiation of measures. We prove that BCP is satisfied by the homogeneous distances whose unit ball centered at the origin coincides with an Euclidean ball. Such homogeneous distances do exist on any Carnot group by a result of Hebisch and Sikora. In the Heisenberg groups, they are related to the Cygan-Kor\'anyi (also called Kor\'anyi) distance. They were considered in particular by Lee and Naor to provide a counterexample to the Goemans-Linial conjecture in theoretical computer science. To put our result in perspective, we also prove two geometric criteria that imply the non-validity of BCP, showing that in some sense our example is sharp. Our first criterion applies in particular to commonly used homogeneous distances on the Heisenberg groups, such as the Cygan-Kor\'anyi and Carnot-Carath\'eodory distances that are already known not to satisfy BCP. To put a different perspective on these results and for sake of completeness, we also give a proof of the fact, noticed by D. Preiss, that in a general metric space, one can always construct a bi-Lipschitz equivalent distance that does not satisfy BCP.
\end{abstract}

\maketitle

\section{Introduction} \label{section:introduction}

Covering theorems are known to be among some of the fundamental tools of measure theory. They reflect the geometry of the space and are commonly used to establish connections between local and global behavior of measures. Covering theorems and their applications have been studied for example in \cite{Fed} and \cite{HP}. There are several types of covering results, all with the same purpose: from an arbitrary cover of a set in a metric space, one extracts a subcover as disjointed as possible. We will consider more specifically here the so-called Besicovitch Covering Property (BCP) which originates from the work of Besicovitch (\cite{B1}, \cite{B2}, see also \cite[2.8]{Fed}, \cite{M}, \cite{P}) in connection with the theory of differentiation of measures. See Subsection~\ref{subsection:bcp} for a more detailed presentation of the Besicovitch Covering Property and its applications.

The geometric setting in which we are interested is the setting of Carnot groups equipped with so-called homogeneous distances, and more specifically here the Heseinberg groups $\HH^n$. Our main result in this paper, Theorem~\ref{thm:main}, is the fact that BCP holds for the homogeneous distances on $\HH^n$ whose unit ball centered at the origin coincides with an Euclidean ball centered at the origin.

These distances are bi-Lipschitz equivalent to any other homogeneous distance on $\HH^n$ and in particular to the commonly used Cygan-Kor\'anyi (also usually called Kor\'anyi or gauge\footnote{We adopt here the terminology \textit{Cygan-Kor\'anyi} distance, that may not be standard, to emphasize the fact that Cygan \cite{cygan} first observed that the natural gauge in the Heisenberg groups actually induces a distance, following in that sense Kor\'anyi \cite{koranyi} who also attribute this distance to Cygan.}) and Carnot-Carath\'eodory distances. Recall that two distances $d$ and $\overline{d}$ are said to be bi-Lipschitz equivalent if there exists $C>1$ such that $C^{-1} d \leq \overline{d} \leq C d$. Cygan-Kor\'anyi and Carnot-Carath\'eodory distances are known not to satisfy BCP (\cite{KoranyiReimann}, \cite{SawyerWheeden}, \cite{Rigot}). We indeed stress that the validity of BCP depends strongly on the distance the space is endowed with, and more specifically on the geometry of its balls. To put some more evidence on this fact and to put our result in perspective, we also prove in the present paper two criteria that imply the non-validity of BCP. They give two large families of homogeneous distances on $\HH^n$ that do not satisfy BCP
and show that in some sense our example for which BCP holds is sharp. See Section~\ref{section:distancesw/obcp}, Theorem~\ref{thm:nobcpingoingcorners} and Theorem~\ref{thm:nobcpoutgoingcorners}. 

As a matter of fact, our first criterion applies to the Cygan-Kor\'anyi and to the Carnot-Carath\'eodory distance, thus giving also new geometric proofs of the failure of BCP for these distances. It also applies to the so-called box-distance (the terminology might not be standard although this distance is a standard homogeneous distance on $\HH^n$, see \eqref{e:box-dist}) thus proving the non-validity of BCP for this latter homogeneous distance as well.

Going back to the distances considered in the present paper and for which we prove that BCP holds, Hebisch and Sikora showed in \cite{Hebisch_Sikora} that in any Carnot group, there are homogeneous distances whose unit ball centered at the origin coincides with an Euclidean ball centered at the origin with a small enough radius. In the specific case of the Heisenberg groups, these distances are related to the Cygan-Kor\'anyi distance. They can indeed be expressed in terms of the quadratic mean of the Cygan-Kor\'anyi distance (at least for some specific value of the radius of the Euclidean ball which coincides with the unit ball centered at the origin) together with the pseudo-distance on $\HH^n$ given by the Euclidean distance between horizontal components. 

These distances have been previously considered in the literature.  Lee and Naor proved in \cite{lee-naor} that these metrics are of negative type on $\HH^n$. Recall that a metric space $(M,d)$ is said to be of negative type if $(M,\sqrt{d})$ is isometric to a subset of a Hilbert space. Combined with the work of Cheeger and Kleiner \cite{cheeger-kleiner} about weak notion of differentiability for maps from $\HH^n$ into $L^1$, which leads in particular to the fact that $\HH^n$ equipped with a homogeneous distance does not admit a bi-Lipschitz embedding into $L^1$, this provides a counterexample to the Goemans-Linial conjecture in theoretical computer science, which was the motivation for these papers. Let us remark that the Cygan-Kor\'anyi distance is not of negative type on $\HH^n$.

We refer to Subsection~\ref{subsection:heisenberg} for the precise definition of our distances and their connection with the Cygan-Kor\'anyi distance and the distances of negative type considered in \cite{lee-naor}.

\subsection{Besicovitch Covering Property} \label{subsection:bcp} Let $(M,d)$ be a metric space. When speaking of a ball $B$ in $M$, it will be understood in this paper that $B$ is a closed ball and that it comes with a fixed center and radius (although these in general are not uniquely determined by $B$ as a set). Thus $B=B_d(p,r)$ for some $p\in M$ and some $r>0$ where $B_d(p,r) = \{q\in M;\; d(q,p)\leq r\}$.

\begin{definition}[Besicovitch Covering Property] One says that the Besicovitch Covering Property (BCP) holds for the distance $d$ on $M$ if there exists an integer $N\geq 1$ with the following property. Let $A$ be a bounded subset of $(M,d)$ and let $\mathcal{B}$ be a family of balls in $(M,d)$ such that each point of $A$ is the center of some ball of $\mathcal{B}$. Then there is a subfamily $\mathcal{F}\subset \mathcal{B}$ whose balls cover $A$ and such that every point in $M$ belongs to at most $N$ balls of $\mathcal{F}$, that is,
\begin{equation*} 
\chi_A \leq \sum_{B\in \mathcal{F}} \chi_B  \leq N,
\end{equation*}
where $\chi_A$ denotes the characteristic function of a set $A$.
\end{definition}

When equipped with a homogeneous distance, the Heisenberg groups turn out to be  doubling metric spaces. Recall that this means that there exists an integer $C\geq 1$ such that each ball with radius $r>0$ can be covered with less than  $C$ balls with radius $r/2$. When $(M,d)$ is a doubling metric space, BCP turns out to be equivalent to a covering property, strictly weaker in general, that we call the Weak Besicovitch Covering Property (w-BCP) (the terminology might not be standard) and with which we shall work in this paper. First, let us fix some more terminology with the following definition.

\begin{definition}[Family of Besicovitch balls]\label{Besicovitch balls}
We say that a family $\mathcal{B}$ of balls in $(M,d)$ is a {\em  family of Besicovitch balls} if $\mathcal{B} = \{B=B_d(x_B,r_B)\}$ is a finite family of balls such that $x_B \not \in B'$ for all $B$, $B'\in \mathcal{B}$, $B\not=B'$, and for which $\bigcap_{B\in \mathcal{B}} B \not= \emptyset$.
\end{definition} 

\begin{definition}[Weak Besicovitch Covering Property] One says that the Weak Besicovitch Covering Property (w-BCP) holds for the distance $d$ on $M$ if there exists an integer $N\geq 1$ such that $\card \mathcal{B} \leq N$ for every family $\mathcal{B}$ of Besicovitch balls in $(M,d)$.
\end{definition} 

The validity of BCP implies the validity of w-BCP. We stress that there exists metric spaces for which w-BCP holds although BCP is not satisfied. However, when the metric is doubling, both covering properties turn out to be equivalent as stated in Characterization~\ref{characterization:equivalentbcp} below.  This characterization can be proved following the arguments of the proof of Theorem 2.7 in \cite{mattila}.

\begin{characterization}[BCP in doubling metric spaces] \label{characterization:equivalentbcp}
Let $(M,d)$ be a doubling metric space. Then BCP holds for the distance $d$ on $M$ if and only if w-BCP holds for the distance $d$ on $M$.
\end{characterization}

As already said, covering theorems and especially the Besicovitch Covering Property and the Weak Besicovitch Covering Property play an important role in many situations in measure theory, regularity and differentiation of measures, as well as in many problems in Harmonic Analysis. This is particularly well illustrated by the connection between w-BCP and the so-called Differentiation theorem. The validity of BCP in the Euclidean space is due to Besicovitch and was a key tool in his proof of the fact that the  Differentiation theorem holds for each locally finite Borel measure on $\R^n$ (\cite{B1}, \cite{B2}, see also \cite[2.8]{Fed}, \cite{M}). Moreover, as emphasized in Theorem~\ref{thm:diff}, the validity of w-BCP  turns actually out to be equivalent to the validity of the Differentiation theorem for each locally finite Borel measure as shown in \cite{P}.

\begin{theorem} \cite[Preiss]{P} \label{thm:diff}
Let $(M,d)$ be a complete separable metric space. Then the Differentiation theorem holds for each locally finite Borel measure $\mu$ on $(M,d)$, that is,
\begin{equation*}
\lim_{r\rightarrow 0^+} \frac{1}{\mu(B_d(p,r))} \int_{B_d(p,r)} f(q) \, d\mu(q) = f(p)
\end{equation*}
for $\mu$-almost every $p\in M$ and for each $f\in L^1(\mu)$ if and only if $M=\cup_{n \in \N} M_n$ where, for each $n\in\N$, w-BCP holds for family of balls centered on $M_n$ with radii less than $r_n$ for some $r_n >0$.
\end{theorem}

As already stressed, the fact that BCP holds in a metric space depends strongly on the distance with which the space is endowed. On the one hand, with very mild assumptions on the metric space (namely, as soon as there exists an  accumulation point), one can indeed always construct bi-Lipschitz equivalent distances as close as we want from the original distance and for which BCP is not satisfied, as shown in the following result. 

\begin{theorem} \label{thm:destroybcp} 
Let $(M,d)$ be a metric space. Assume that there exists an accumulation point in $(M,d)$. Let $0<c<1$. Then there exists a distance $\overline{d}$ on $M$ such that $c \, d \leq \overline{d} \leq d$ and for which w-BCP, and hence BCP, do not hold.
\end{theorem}

A slightly different version of this result is stated in Theorem 3 of \cite{P}. For sake of completeness, we  give in Section~\ref{section:destroybcp} a construction of such a distance as stated in Theorem~\ref{thm:destroybcp}.

On the other hand, the question whether a metric space can be remetrized so that BCP holds is in general significantly more delicate. As already explained, the main result of the present paper, Theorem~\ref{thm:main}, is a positive answer to this question for the Heisenberg groups equipped with ad-hoc homogeneous distances, namely those whose unit ball at the origin coincides with an Euclidean ball with a small enough radius. 

\medskip

\subsection{The Heisenberg group} \label{subsection:heisenberg} As a set we identify the Heisenberg group $\HH^n$ with $\R^{2n+1}$ and we equip it as a topological space with the Euclidean topology. We choose the following convention for the group law 
\begin{equation} \label{e:grouplaw}
(x,y,z)\cdot(x',y',z'):=(x+x',y+y',z+z'  +\dfrac{1}{2}\langle x,y' \rangle -\dfrac{1}{2}\langle y,x'\rangle)
\end{equation}
where $x$, $y$, $x'$ and $y'$ belong to $\mathbb{R}^n$, $z$ and $z'$ belong to $\R$ and $\langle \cdot,\cdot \rangle $ denotes the usual scalar product in $\R^n$. This corresponds to a choice of exponential and homogeneous coordinates. 

The one parameter family of dilations on $\HH^n$ is given by $(\delta_\lambda)_{\lambda>0}$ where 
\begin{equation} \label{e:dilations}
\delta_\lambda (x,y,z) := (\lambda x, \lambda y, \lambda^2 z).
\end{equation}
These dilations are group automorphisms.

\begin{definition}[Homogeneous distance] \label{def:homogeneousdistance}
A distance $d$ on $\HH^n$ is said to be {\em homogeneous} if the following properties are satisfied. First, it induces the Euclidean topology on $\HH^n$. Second, it is left invariant, that is, $d(p\cdot q, p\cdot q') = d(q,q')$ for all $p$, $q$, $q' \in \HH^n$. And third, it is one-homogeneous with respect to the dilations, that is, $d(\delta_\lambda(p), \delta_\lambda(q)) = \lambda\; d(p,q)$ for all $p$, $q\in\HH^n$ and all $\lambda >0$.
\end{definition}

It turns out that homogeneous distances on $\HH^n$ do exist in abundance and make it a doubling metric space. It is also well known that any two homogeneous distances are bi-Lipschitz equivalent. See for example \cite{folland-stein} for more details about the Heisenberg groups and more generally Carnot groups.

The (family of) homogeneous distance(s) we consider in this paper can be defined in the following way. For $\alpha>0$, we denote by $B_\alpha$ the Euclidean ball in $\HH^n \simeq \R^{2n+1}$ centered at the origin with radius $\alpha$, that is, $$B_\alpha := \{(x,y,z)\in \HH^n;\; \| x\|_{\R^{n}}^2+\| y\|_{\R^{n}}^2+|z|^2 \leq \alpha^2\},$$
where $\| \cdot \|_{\R^{n}}$ denotes the Euclidean norm in $\R^{n}$ and we set 
\begin{equation} \label{e:defdistdalpha}
d_\alpha(p,q) := \inf\{r>0; \; \delta_{1/r}(p^{-1}\cdot q) \in B_\alpha \}\,.
\end{equation}

Hebisch and Sikora proved in \cite{Hebisch_Sikora} that if $\alpha>0$ is small enough, then $d_\alpha$ actually defines a distance on $\HH^n$. More generally this holds true in any Carnot group starting from the set $B_\alpha$ given by the Euclidean ball centered at the origin with radius $\alpha>0$ small enough, where one identifies in the usual way the group with some $\R^m$ where $m$ is its topological dimension. 

It then follows from the very definition that $d_\alpha$ turns out to be the homogeneous distance on $\HH^n$ for which the unit ball centered at the origin coincides with the Euclidean ball with radius $\alpha$ centered at the origin. The geometric description of arbitrary balls that can then be deduced from the unit ball centered at the origin via dilations and left-translations is actually of crucial importance for understanding the reasons why BCP eventually holds for these distances.

On the other hand, it is particularly convenient to note that in the specific case of the Heisenberg groups, one also has a fairly simple analytic expression for such distances whose unit ball at the origin is given by an Euclidean ball centered at the origin. This will actually be technically  extensively used in our proof of Theorem~\ref{thm:main}. This also gives the explicit connection with the Cygan-Kor\'anyi distance and the distances of negative type considered by Lee and Naor in \cite{lee-naor}. 

Set 
\begin{equation} \label{e:n-korany-norm}
\rho(p):=\sqrt{\| x\|_{\R^{n}}^2+\| y\|_{\R^{n}}^2} \quad \text{and} \quad \|p\|_{g,\alpha}:= \left(\rho(p)^4 + 4 \alpha^2 |z|^2\right)^{1/4}
\end{equation}
for $p=(x,y,z)\in\HH^n$. Then one has 
\begin{equation} \label{e:analyticdalpha}
d_\alpha(p,q) = \sqrt{ \dfrac{ \rho(p^{-1}\cdot q)^2 + \|p^{-1}\cdot q\|_{g,\alpha}^2}{2\alpha^2} }\,,
\end{equation}
see Section~\ref{section:preliminaries}. 

First, note that $d_\rho(p,q):= \rho(p^{-1}\cdot q)$ is a left-invariant pseudo-distance on $\HH^n$ that is one-homogeneous with respect to the dilations. Next, when $\alpha=2$, $\|\cdot\|_{g,2}$ is nothing but the Cygan-Kor\'anyi norm which is well known to be a natural gauge in $\HH^n$. It can actually be checked by direct computations that $d_{g,\alpha}(p,q):= \|p^{-1}\cdot q\|_{g,\alpha}$ satisfies the triangle inequality for any $0<\alpha\leq 2$ and hence defines a homogeneous distance on $\HH^n$. This was first proved by Cygan in \cite{cygan} when $\alpha=2$. One then recovers from the analytic expression \eqref{e:analyticdalpha} that $d_\alpha$ actually defines a homogeneous distance on $\HH^n$ for any $0<\alpha\leq 2$, giving also an explicit range of values of $\alpha$ in $\HH^n$ for which this fact holds and was first observed in \cite{Hebisch_Sikora} for general Carnot groups and for small enough values of $\alpha$. 

\begin{theorem} For any $0<\alpha \leq 2$, $d_\alpha$ defines a homogeneous distance on $\HH^n$. 
\end{theorem}

Note that there might be other values of $\alpha>2$ such that $d_\alpha$ defines a homogeneous distance on $\HH^n$. 

These distances turn out to be those considered by Lee and Naor in \cite{lee-naor}. The authors actually proved in \cite{lee-naor} that $d_2$ is of negative type in $\HH^n$ to provide a counterexample to the so-called Goemans-Linial conjecture. Let us mention that it can easily be checked that the proof in \cite{lee-naor} extend to the distances $d_\alpha$ for all $0<\alpha \leq 2$.

Let us now state our main result.

\begin{theorem} \label{thm:main}
Let $\alpha>0$ be such that $d_\alpha$ defines a homogeneous distance on $\HH^n$. Then BCP holds for the homogeneous distance $d_\alpha$ on $\HH^n$.
\end{theorem}

For technical and notational simplicity, we will focus our attention on the first Heisenberg group $\HH=\HH^1$. We shall point out briefly in Section~\ref{section:Hn} the non-essential modifications needed to make our arguments work in any Heisenberg group $\HH^n$. 

The rest of the paper is organized as follows. In Section~\ref{section:preliminaries} we fix some conventions about $\HH$ and the distance $d_\alpha$ and state three technical lemmas on which the proof of Theorem~\ref{thm:main} is based. The proof of these lemmas is given in 
Sections \ref{section:prooflemma:axis} and
 \ref{section:prooflemma:comparisonincone}. Section~\ref{section:proofmainthm} is devoted to the proof of Theorem~\ref{thm:main} itself. In Section~\ref{section:distancesw/obcp} we prove two criteria, Theorem~\ref{thm:nobcpingoingcorners} and Theorem~\ref{thm:nobcpoutgoingcorners}, for homogeneous distances on $\HH$ that imply that BCP does not hold. %The modifications needed to generalize all these results to any Heisenberg group $\HH^n$ are indicated in Section~\ref{section:Hn}. 
Theorem~\ref{thm:destroybcp} is proved in Section~\ref{section:destroybcp}.

\section{Preliminary results} \label{section:preliminaries}

As already stressed we will focus our attention in Sections \ref{section:preliminaries} to \ref{section:distancesw/obcp} on the first Heisenberg group $\HH=\HH^1$ for technical and notational simplicity. The modifications needed to handle the case of $\HH^n$ for any $n\geq 1$ will be indicated in Section~\ref{section:Hn}.

We first fix some conventions and notations. Next, we will conclude this section with the statement of the main lemmas on which the proof of Theorem~\ref{thm:main} will be based.

Recall that we identify the Heisenberg group $\HH$ with $\R^3$ equipped with the group law given in \eqref{e:grouplaw} and we equip it with the Euclidean topology.

We define the projection $\pi: \HH \rightarrow \R^2$ by 
\begin{equation} \label{e:projection}
\pi(x,y,z) := (x,y).
\end{equation}

When considering a specific point $p\in\HH$, we shall usually denote by $(x_p,y_p,z_p)$ its coordinates and we set
\begin{equation} \label{e:rhop}
\rho_p := \sqrt{x_p^2+ y_p^2}\,. 
\end{equation}

From now on in this section, as well as in Sections \ref{section:proofmainthm}, \ref{section:prooflemma:axis} and \ref{section:prooflemma:comparisonincone}, we fix some $\al>0$ such that $d_\al$ as given in \eqref{e:defdistdalpha} defines a homogeneous distance on $\HH$. Thus all metric notions and properties will be understood in these sections relatively to this fixed distance $d_\al$. In particular we shall denote the closed balls with center $p\in\HH$ and radius $r>0$ by $B(p,r)$ without further explicit reference to the distance $d_\al$ with respect to which they are defined.

Remembering \eqref{e:defdistdalpha}, we have the following properties.

\begin{proposition}
For $p=(x_p,y_p,z_p)\in \HH$, we have
\begin{equation} \label{formula:norm0}
d_\alpha(0,p)\leq r  \iff 
\dfrac{\rho_p^2}{r^2}+\dfrac{z_p^2}{r^4}\leq\alpha^2
 \end{equation}
 and
\begin{equation} \label{e:norm0'}
d_\alpha(0,p) = r  \iff 
\dfrac{\rho_p^2}{r^2}+\dfrac{z_p^2}{r^4} = \alpha^2
\end{equation}
from which we get 
\begin{equation} \label{formula:norm1}
d_\alpha(0,p) = \sqrt{ \dfrac{ \rho_p^2 +\sqrt{\rho_p^4 + 4 \alpha^2 z_p^2} }{2\alpha^2} }~~.
\end{equation}
\end{proposition}

For a point $p\in\HH$, we shall set  
\begin{equation} \label{e:rp}
r_p := d_\al(0,p)\, .
\end{equation}

Using left-translations, we have the following properties for any two points $p,q\in\HH$,
\begin{equation} \label{e:dalpha1}
d_\alpha(p,q)\leq r  \iff 
\dfrac{\rho_{p^{-1}\cdot q}^2}{r^2}+\dfrac{z_{p^{-1}\cdot q}^2}{r^4}\leq\alpha^2
\end{equation}
and
\begin{equation} \label{e:dalpha2}
d_\alpha(p,q) = \sqrt{ \dfrac{ \rho_{p^{-1}\cdot q}^2 +\sqrt{\rho_{p^{-1}\cdot q}^4 + 4 \;\alpha^2\; z_{p^{-1}\cdot q}^2} }{2\alpha^2} }
\end{equation}
where $$\rho_{p^{-1}\cdot q} = \sqrt{(x_q - x_p)^2+ (y_q -y_p)^2}$$ and $$z_{p^{-1}\cdot q} = z_q -z_p - \dfrac{x_p y_q - y_p x_q}{2}$$ by definition of the group law \eqref{e:grouplaw}. Note that if $p=(x_p,y_p,z_p)\in\HH$ then $p^{-1} = (-x_p,-y_p,-z_p)$.

Let us point out that balls in $(\HH,d_\al)$ are convex in the Euclidean sense when identifying $\HH$ with $\R^3$ with our choosen coordinates. Indeed, the unit ball centered at the origin is by definition the Euclidean ball with radius $\alpha$ in $\HH\simeq\R^3$ and thus is Euclidean convex. Next, dilations \eqref{e:dilations} are linear maps and left-translations (see \eqref{e:grouplaw}) are affine maps, hence 
$$B(p,r)= p \cdot \delta_r(B(0,1))$$
is also an Eucliden convex set in $\HH\simeq\R^3$. This will be of crucial use for some of our arguments in the sequel and we state it below as a proposition for further reference.

\begin{proposition} \label{prop:ball_convex}
Balls in $(\HH,d_\al)$ are convex in the Euclidean sense when identifying $\HH$ with $\R^3$ with our choosen coordinates.
\end{proposition}

We shall also use the following isometries of $(\HH,d_\al)$. First, rotations  around the $z$-axis are defined by
\begin{equation} \label{e:rotations}
\operatorname{R}_\theta : (x,y,z) \mapsto (x \cos\th - y\sin\th,x\sin\th+y\cos\th,z)
\end{equation}
for some angle $\th\in\R$. Next, the reflection $\operatorname{R}$ is defined by 
\begin{equation} \label{e:reflection}
\operatorname{R} (x,y,z) := (x,-y,-z).
\end{equation}
Using \eqref{e:dalpha2}, one can easily check that these maps are isometries of $(\HH,d_\al)$.

\medskip

We state now the main lemmas on which the proof of Theorem~\ref{thm:main} will be based.

For $\th\in(0,\pi/2)$, $a>0$ and $b>0$, we set (see  Figure \ref{fig1})
\begin{equation} \label{e:defP} 
\PP(a,b,\theta):=\{p\in \HH; \; x_p>a, \; |z_p|<b, \; |y_p|<x_p \tan \theta \}.
\end{equation}

   \begin{figure}
\centering
\subfigure%[The truncated cilinder $\mathcal T (a,b)$] % caption for subfigure a
{
    \label{fig:P}
 \includegraphics[height=6cm]{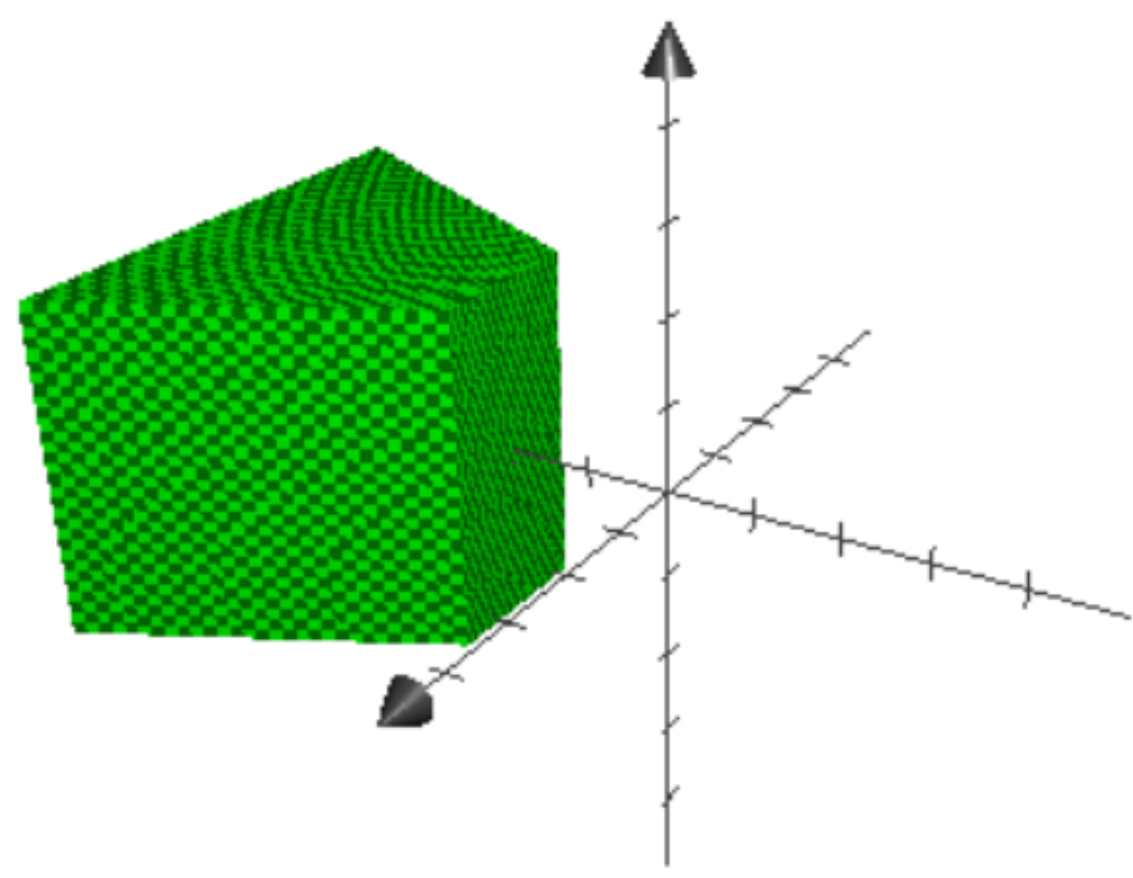}
}
% \hspace{1cm}  
\subfigure%[The conic sector $\mathcal C(\theta)$.] % caption for subfigure b
{
    \label{fig:C_theta}
\includegraphics[height=6cm]{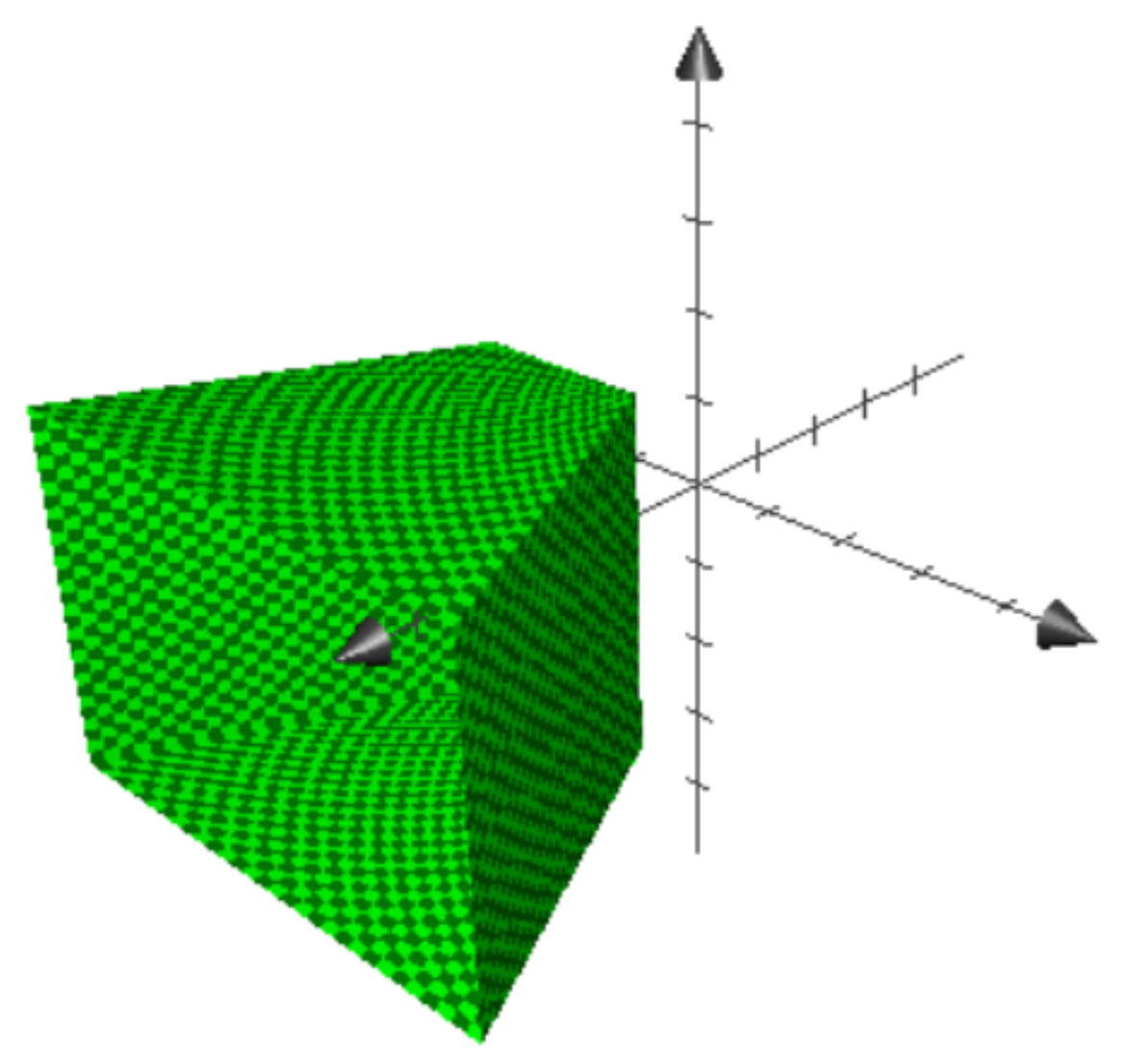}
}
\caption{Two views of the region $\mathcal P(a, b,\theta)$.}
\label{fig1}
%\label{special segments} % caption for the whole figure
\end{figure}

\begin{lemma} \label{lemma:x:axis0}
There exists $\theta_0\in (0,\pi/4)$, which depends only on $\al$, such that for all $\th\in (0,\th_0)$, there exists $a_0(\theta)\geq 1$ such that for all $a>a_0(\theta)$ and for all $b\in (0,1)$, the following holds. Let $p\in \HH$ and $q\in \HH$ be such that $p\notin  B(q, r_q)$ and $q\notin B(p, r_p)$. Then at most one of these two points belongs to  $\PP(a,b, \theta)$.
\end{lemma}

For $a>0$ and $b>0$, we set (see  Figure \ref{fig2}(a))
\begin{equation} \label{e:defT}
\TT(a,b):=\{p\in \HH ; \; z_p<-a, \;\rho_p<b\}.
\end{equation}

   \begin{figure}
\centering
\subfigure[The truncated cylinder $\mathcal T (a,b)$] % caption for subfigure a
{
    \label{fig:T}
 \includegraphics[height=7cm]{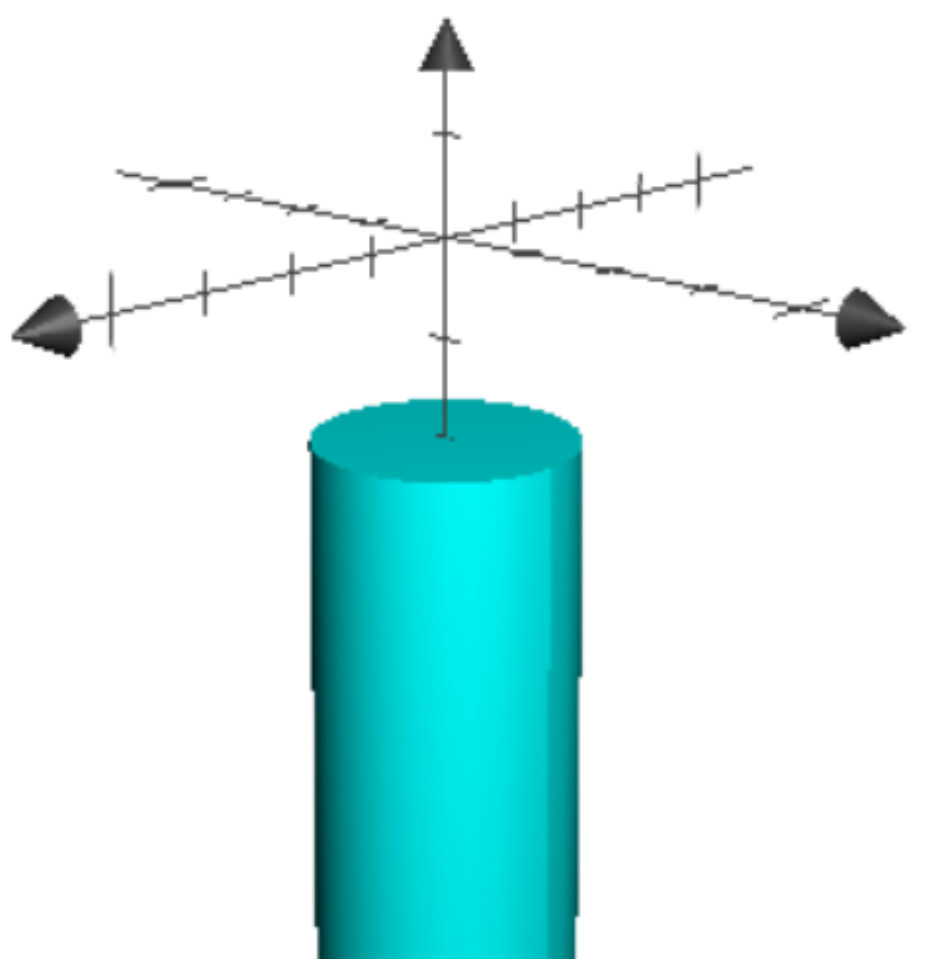}
}
% \hspace{1cm}  
\subfigure[The conic sector $\mathcal C(\theta)$.] % caption for subfigure b
{
    \label{fig:C}
\includegraphics[height=6cm]{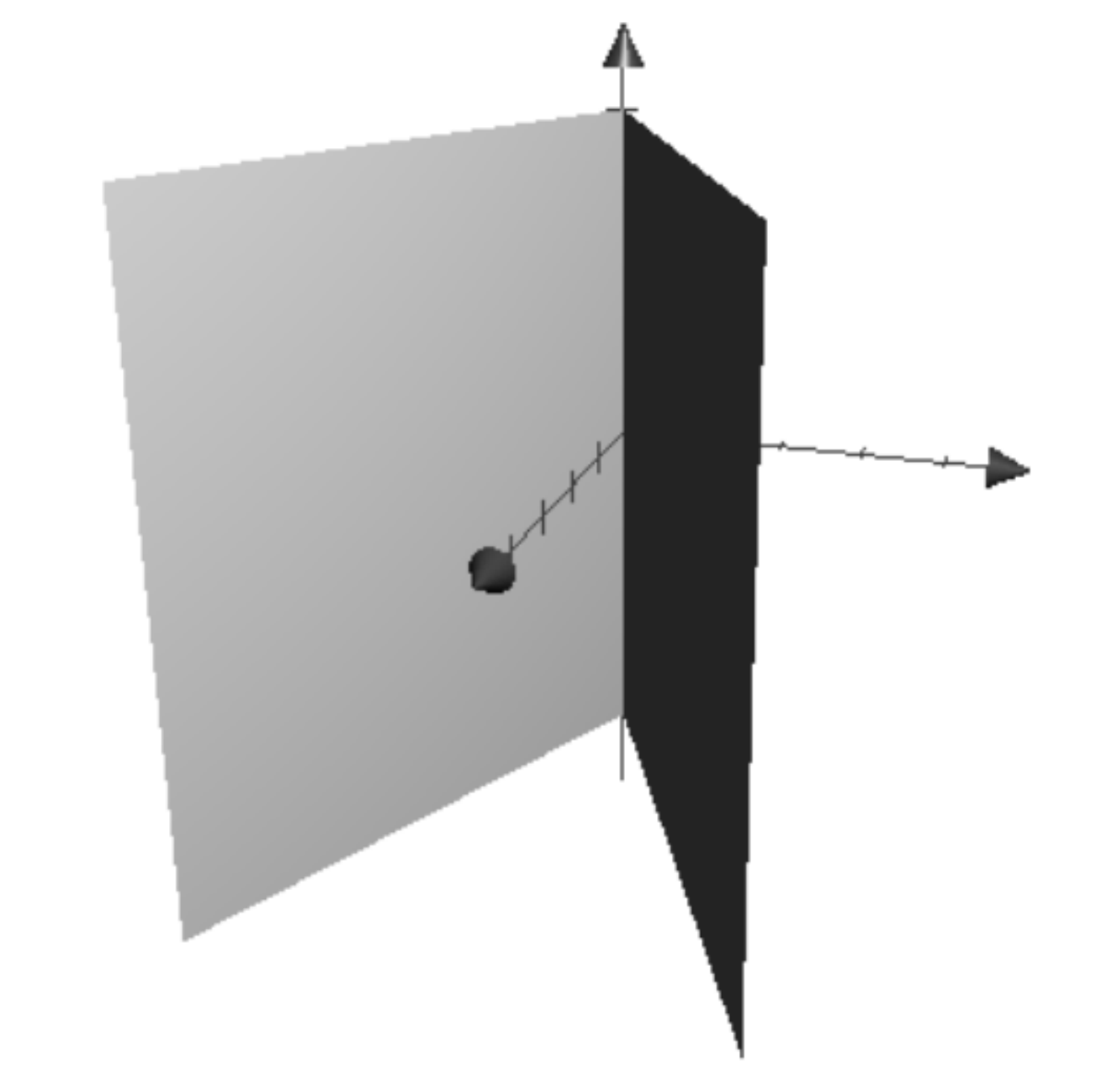}
}
\caption{The regions $\mathcal T (a,b)$ and  $\mathcal C(\theta)$.}
\label{fig2}
%\label{special segments} % caption for the whole figure
\end{figure} 
\begin{lemma} \label{lemma:z:axis0}
There exists $a_1 \geq 1$ and $b_1 \in (0,1)$, depending only on $\al$, such that for all $a>a_1$ and all $b\in (0,b_1)$, the following holds. Let $p\in \HH$ and $q\in \HH$ be such that $p\notin  B(q, r_q)$ and $q\notin B(p, r_p)$. Then at most one of these two points belongs to  $\TT(a,b)$.
\end{lemma}

These two lemmas will be proved in Section~\ref{section:prooflemma:axis}.

For $\th \in(0,\pi/2)$, we set (see  Figure \ref{fig2}(b))
\begin{equation} \label{e:defCtheta}
\mathcal{C}(\theta): = \{p\in\HH;\, \abs{y_p} <  x_p \tan \theta \}.
\end{equation} 

\begin{lemma} \label{lemma:comparisonincone}
There exists $\th_1 \in (0,\pi/8)$, which depends only on $\al$, such that for all $\th\in (0,\th_1)$ the  following holds. Let $p\in\HH$ and $q\in\HH$ be such that
\begin{gather}
z_q \leq 0 \text{ and } z_p \leq 0 \label{e:c1}\\
\rho_q \leq \rho_p \label{e:c2}\\
q\in \mathcal{C}(\theta) \text{ and }  p\in \mathcal{C}(\theta) \label{e:c3}\\
q\not\in B(p,r_p)  \text{ and } p\not\in B(q,r_q) .\label{e:c4}
\end{gather}
Then we have
\begin{equation}
z_q<2 \, z_p \label{e:c5}
\end{equation}
and
\begin{equation}
\rho_q < \rho_p \cos(2\theta). \label{e:c6}
\end{equation}
\end{lemma}

This lemma will be proved in Section~\ref{section:prooflemma:comparisonincone}.

\section{Proof of Theorem~\ref{thm:main}} \label{section:proofmainthm}

This section is devoted to the proof of Theorem~\ref{thm:main}. Recall that we consider here the case $\HH = \HH^1$ equipped with a homogeneous distance $d_\al$ as defined in \eqref{e:defdistdalpha} (see Section~\ref{section:Hn} for the general case $\HH^n$, $n\geq 1$). Recall also that due to Characterization~\ref{characterization:equivalentbcp}, Theorem~\ref{thm:main} will follow if we find an integer $N\geq 1$ such that $\card \mathcal{B} \leq N$ for every family $\mathcal{B}$ of Besicovitch balls. See Definition~\ref{Besicovitch balls} for the definition of a family of Besicovitch balls.

We first reduce the proof to the case of some specific families of Besicovitch balls. In what follows, when considering families of points $\{p_j\}$ we shall simplify the notations and set $p_j = (x_j,y_j,z_j)$, $\rho_j = \sqrt{x_j^2+ y_j^2}$ and $r_j = d_\al(0,p_j)$. Recall that $\mathcal{C}(\th)$ is defined in \eqref{e:defCtheta}.

\begin{lemma}\label{lemma:reduction}
Let $\th \in (0,\pi/2)$ and let $\mathcal B$ be a family of Besicovitch balls. Then there exists a finite family of points $\{p_j\}$ such that $\mathcal{F} = \{B(p_j, r_j)\}$ is a family of Besicovitch balls with the following properties. For every point $p_j$ in the family, we have
\begin{gather}
z_j \leq0,\label{e:lemreduc1} \\
p_j \in \mathcal{C}(\th), \label{e:lemreduc2}
\end{gather}
and
\begin{equation} \label{e:lemreduccard}
\card \mathcal{B} \leq 2 \left(\dfrac{\pi}{\th} +1\right) \card \mathcal{F} +2.
\end{equation}
\end{lemma}

\begin{proof}
Let $\mathcal B = \{ B(q_j,t_j)\}_{j=1}^k$ be a family of Besicovitch balls where $k=\card \mathcal B$.
 Take $q\in \cap_{j=1}^k B(q_j, t_j)$. Set $p_j = q^{-1}\cdot q_j$. Remembering that left-translations are isometries and that, by convention, we set $r_j = d_\al(0,p_j)$, we get that $0\in \cap_{j=1}^k B(p_j, r_j)$ and $d_\al(p_j,p_i) =d_\al(q_j,q_i) >\max(t_j,t_i) \geq \max(r_j,r_i)$ hence $\mathcal B' =\{ B(p_j, r_j)\}_{j=1}^k$ is a family of Besicovitch balls.

Since balls are Euclidean convex (see Proposition \ref{prop:ball_convex}) and since $0\in \partial B(p_j, r_j)$ for all $j=1,\ldots,k$, 
there are at most two balls in $\mathcal B'$ with their center on the $z$-axis.

Next, up to replacing the family $\{p_j\}$ by $\{\operatorname{R}(p_j)\}$ (see \eqref{e:reflection} for the definition of the reflection $\operatorname{R}$) and up to re-indexing the points, one can find $l$ points $p_1, \ldots, p_l$ that satisfy \eqref{e:lemreduc1}, such that  $\pi(p_1), \ldots, \pi(p_l)\neq 0$ (see \eqref{e:projection} for the definition of the projection $\pi$), and with $2l \geq (k-2)$.

Finally, up to a rotation around the $z$-axis (see \eqref{e:rotations} for the definition of rotations) and up to re-indexing the points, we get by the pigeonhole principle that there exists an integer $k'$ such that $$(\dfrac{\pi}{\th} +1)\; k' \geq l$$ and such that $p_j$ satisfies  \eqref{e:lemreduc2} for all $j=1,\ldots,k'$. Then the family $\mathcal{F} = \{B(p_j, r_j)\}_{j=1}^{k'}$ gives the conclusion.
\end{proof}

We are now ready to conclude the proof of Theorem~\ref{thm:main} using Lemma~\ref{lemma:x:axis0},
Lemma~\ref{lemma:z:axis0} and 
Lemma~\ref{lemma:comparisonincone}.

\medskip

\noindent \textbf{\textit{Proof of Theorem~\ref{thm:main}.}} We fix some values of $\th\in (0,\pi/8)$, $a>0$, and $b>0$ so that the conclusions of Lemma~\ref{lemma:x:axis0}, Lemma~\ref{lemma:z:axis0} and Lemma~\ref{lemma:comparisonincone} hold.

Next, we fix some $R>0$ large enough so that 
$$ \{p\in \HH;\; x_p\in[0,a], \;|z_p|<b,\;  |y_p|<x_p \tan \theta\} 
\subset
 U(0,R)$$
and
$$\{p\in\HH;\;  z_p\in[-a,0],\; \rho_p<b\}
\subset
 U(0,R),$$
where $U(0,R)$ denotes the open ball with center 0 and radius $R$ in $(\HH,d_\alpha)$. Such an $R$ exists since in the above two inclusions, the sets on the left are bounded.
As a consequence, we have
\begin{equation}
\label{condizione1}
(\HH\setminus U(0,R)) \cap \{p\in\HH;\; |z_p|<b,\;  |y_p|<x_p  \tan \theta \} \subset \PP(a,b,\th)
\end{equation}
and
\begin{equation}
\label{condizione2}
(\HH\setminus U(0,R)) \cap \{p\in \HH;\;  z_p \leq 0 ,\; \rho_p<b\} \subset \TT(a,b),
\end{equation}
recall \eqref{e:defP} for the definition of $\PP(a,b,\th)$ and \eqref{e:defT} for the definition of $\TT(a,b)$.

 Let us now consider a family of Besicovitch balls $\mathcal{F} = \{B(p_j, r_j)\}_{j=1}^k$ where, as defined by convention, we have $r_j = d_\al(0,p_j)$ and where the centers $p_j$ satisfy \eqref{e:lemreduc1} and \eqref{e:lemreduc2}. Noting that the family $\{B(\delta_\lambda(p_j), \lambda r_j)\}_{j=1}^k$ also satisfies the same properties for all $\lambda>0$, one can assume with no loss of generality that $$R=\min\{ d_\alpha (0,p_j);\; j=1, \ldots, k\}$$
up to a dilation by a factor $\lambda = R / \min \{r_1,\dots,r_k\} $.

Let $m>0$ and $M>0$ be defined as 
$$-m:= \min \{ z_p;\; p\in B(0,R)\}$$
and
$$M:= \max \{ \rho_p;\; p\in B(0,R)\}.$$
We will bound $k= \card \mathcal F$ in terms of the constants $m$, $M$, $b$ and $\theta$.

We re-index the points so that
$$0< \rho_1\leq \rho_2\leq \ldots \leq \rho_k.$$
Let $l\in\{1, \ldots, k\}$ be such that $d_\al(0,p_l) = R$. By choice of $l$ and by definition of $m$ and $M$,  we have 
$$\rho_l \leq M \quad \text{ and }\quad -m\leq z_l \, .$$

Let $j_0\geq 1$ be a large enough integer such that $M \cos^{j_0} (2\theta) <b$. Then we have $l\leq j_0+1$. Indeed, otherwise we would get from \eqref{e:c6} in Lemma~\ref{lemma:comparisonincone} that
\begin{equation*}
0< \rho_1 < \rho_2 \cos(2\theta) < \cdots < \rho_{l} \cos^{l-1}(2\theta) \leq  M \cos^{j_0+1} (2\theta) <b \cos(2\theta)
\end{equation*}
and hence $\rho_1 < \rho_2 <b$. Then, by choice of $R$ (remember \eqref{condizione2}), $p_1$ and $p_2$ would be distinct points in $\TT(a,b)$ which contradicts Lemma~\ref{lemma:z:axis0}. 

Let $j_1\geq 1$ be a large enough integer such that $2^{-j_1} m<b$. Then we have $k-l\leq j_1$. Indeed, otherwise we would get from \eqref{e:c5} in Lemma~\ref{lemma:comparisonincone} that 
\begin{equation*}
-m \leq z_l < \dots < 2^{k-l-1} z_{k-1} < 2^{k-l} z_{k} \leq 0
\end{equation*}
and hence $|z_k| < |z_{k-1}| < 2^{-(k-l-1)} m \leq 2^{-j_1} m<b$. Then, by choice of $R$ (remember \eqref{condizione1}), $p_{k-1}$ and 
$p_k$ would be  distinct points in $\PP(a,b,\th)$ which contradicts Lemma~\ref{lemma:x:axis0}. 

All together we get the following bound on $\card \mathcal F=k$,
 $$ \card \mathcal F \leq
 \log_2(m/b)+ \log_{\cos(2\theta)}(b/M) +3.$$
 
 Combining this with \eqref{e:lemreduccard} in Lemma~\ref{lemma:reduction}, we get the following bound on the cardinality of arbitrary families $\mathcal B$ of Besicovitch balls,
 $$\card \mathcal B \leq 2 (\pi / \th +1)(\log_2(m/b)+ \log_{\cos(2\theta)}(b/M)+3)  +2\, ,$$
 which concludes the proof of Theorem~\ref{thm:main}.
 \qed

\section{Proof of Lemma~\ref{lemma:x:axis0} and of Lemma~\ref{lemma:z:axis0}} \label{section:prooflemma:axis}

This section is devoted to the proof of Lemma~\ref{lemma:x:axis0} and Lemma~\ref{lemma:z:axis0}. We begin with a remark that will be technically useful. Given $p\in\HH$ and $q\in\HH$, we set
$$A_p(q):=r^2_p
\left(x_q^2 +
y_q^2 - 2 x_q x_p -2 y_q y_p\right)
+
 \left(z_q   - \dfrac{x_p y_q - x_q y_p}{2}  \right  )^2
 -2 z_p\left(z_q   - \dfrac{x_p y_q - x_q y_p}{2}  \right  ).
 $$
Recall that, following \eqref{e:rp}, we have $r_p=d_\alpha(0,p)$ by convention.
 
 \begin{lemma}\label{A(p,q)}
We have $q\in B(p,r_p) $ if and only if $A_p(q)\leq 0.$
\end{lemma}
\begin{proof}
Recalling \eqref{e:dalpha1}, we have 
 $$
d_\alpha(p,q) \leq r_p \iff
\dfrac{(x_q-x_p)^2}{r_p^2} +
\dfrac{(y_q-y_p)^2}{r_p^2} +
\dfrac{\left(z_q-z_p   - \dfrac{x_p y_q - x_q y_p}{2}  \right  )^2}{r_p^4} 
\leq
\alpha^2~~.
$$
Combining this with \eqref{e:norm0'}, which gives
$$\dfrac{x_p^2+y_p^2}{r_p^2} +
\dfrac{z_p   ^2}{r_p^4} 
=
\alpha^2~,$$ we get the conclusion.
\end{proof}

\subsection{Proof of Lemma~\ref{lemma:x:axis0}}

\begin{lemma}\label{lemma:PP}
There exist constants $c_1>0$ and $c_2>0$, depending only on $\al$, such that, 
for all $\theta\in (0,\pi/4)$, all $a>0$ and $b>0$ such that $a^2 \geq b$, we have
$$c_1 \; x_p \leq r_p  \leq c_2 \; x_p$$
for all $p\in \PP(a,b, \theta)$.
\end{lemma}

\begin{proof} By \eqref{formula:norm1}, we always have  $r_p^2 \geq x_p^2  / (2 \alpha^2) $. On the other hand, we can bound from above $r_p^2$ using that $\tan \theta < 1$, since $\theta<\pi/4$, and that $|z_p|<b\leq a^2\leq x_p^2$ if  
$p\in \PP(a,b, \theta)$ (see \eqref{e:defP} for the definition of $\PP(a,b, \theta)$). Namely, we have
\begin{eqnarray*}
r_p^2 &=&
\dfrac{x_p^2 +y_p^2 +\sqrt{ (x_p^2 +y_p^2 )^2+4\alpha^2 z_p^2}}{2\alpha^2} \\
&\leq&
\dfrac{x_p^2 (1+\tan^2\theta)     +\sqrt{ (x_p^2 (1+\tan^2\theta)  )^2+4\alpha^2 z_p^2}}{2\alpha^2} \\
&\leq&
\dfrac{2 x_p^2     +\sqrt{ 4 x_p^4+4\alpha^2 b^2}}{2\alpha^2} \\
&\leq&
\dfrac{2      +\sqrt{ 4  +4\alpha^2  }}{2\alpha^2} \,  x_p^2\, .
\end{eqnarray*}
\end{proof}

For $t\in\R$, $b>0$ and $\th\in(0,\pi/2)$, we set (see  Figure \ref{fig3}(a))
$$\RR(t, b,\theta):=\{p\in \HH;\; x_p=t,\; |z_p|<b,\;|y_p|<x_p \tan \theta \}. $$

\begin{figure}
\centering
\subfigure[The quadrilateral $\mathcal R (t,b,\theta)$] % caption for subfigure a
{
    \label{fig:R_in_P}
 \includegraphics[height=6.5cm]{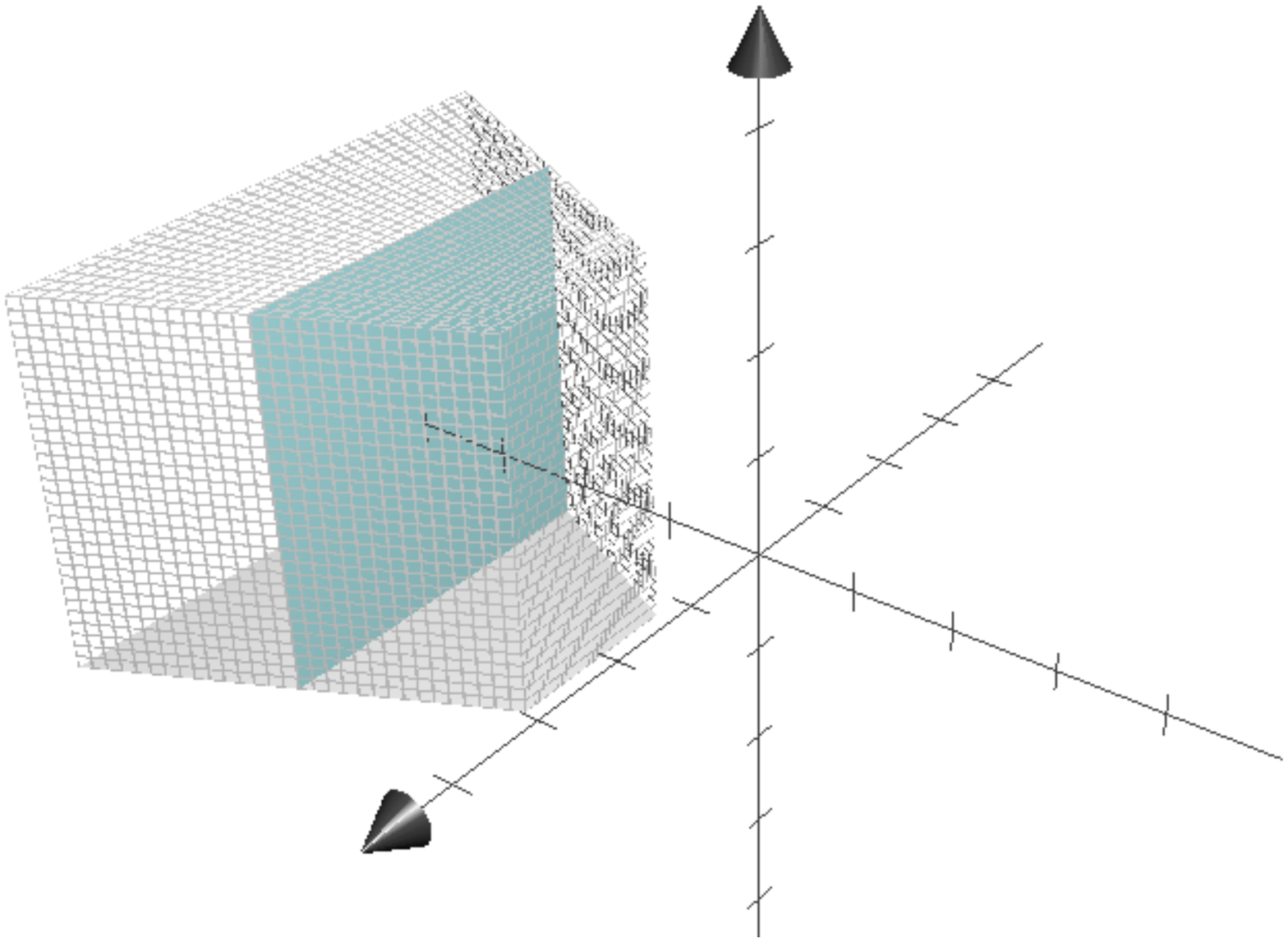}
}
% \hspace{1cm}  
\subfigure[The disc $\mathcal D (t,b)$] % caption for subfigure b
{
    \label{fig:R}
\includegraphics[height=5.5cm]{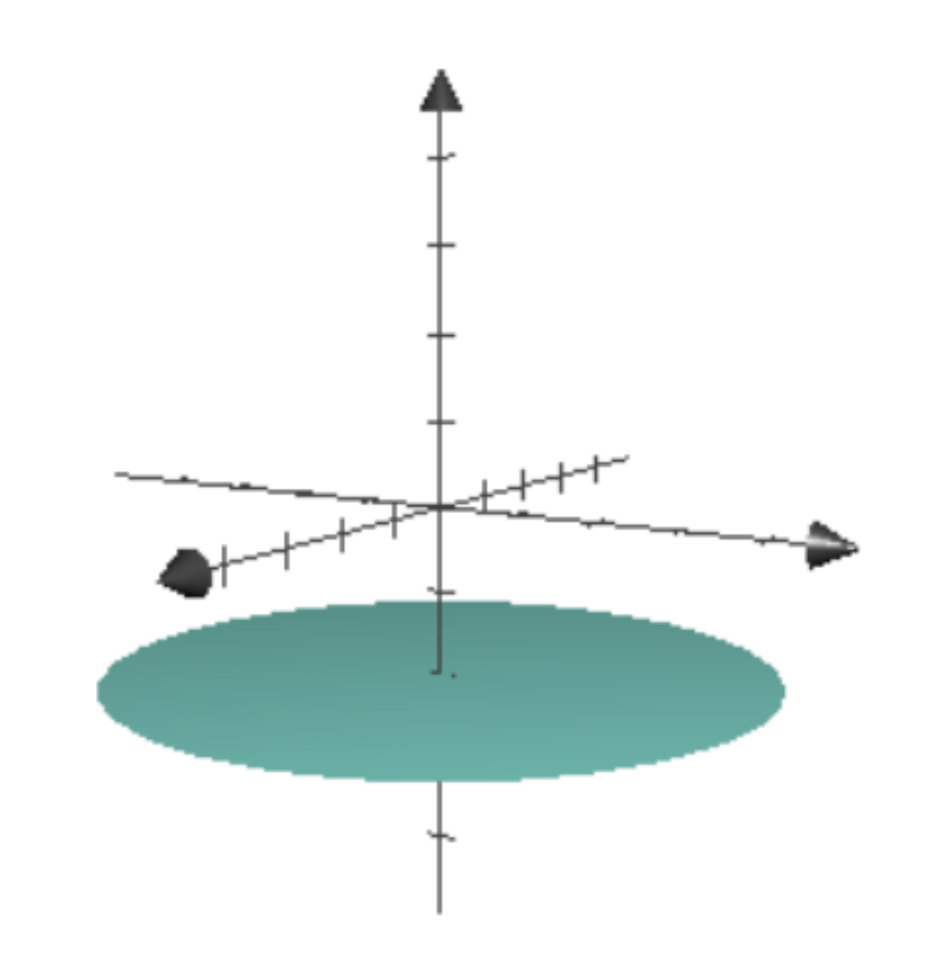}
}
\caption{The surfaces $\mathcal R (t,b,\theta)$ and $\mathcal D (t,b)$.}
\label{fig3}
%\label{special segments} % caption for the whole figure
\end{figure} 

\begin{lemma}\label{lemma:x:axis}
There exists $\theta_0\in (0,\pi/4)$, which depends only on $\al$, such that for all $\theta \in (0,\theta_0)$, there exists $a_0(\theta) \geq 1$
such that for all $a>a_0(\theta)$ and for all $b\in (0,1)$, we have 
$$\RR(t, b,\theta) \subset B(p,r_p)$$
for all  $p\in\PP(a, b, \theta)$ and all $t\in [1,x_p]$.
\end{lemma}

\begin{proof} Take $\theta\in (0,\pi/4)$, $a\geq 1 > b$, $p\in\PP(a, b, \theta)$, $t>0$ and consider $q\in \RR(t, b,\theta)$.
By Lemma~\ref{A(p,q)}, showing that $q\in B   (p,r_p) $ is equivalent to prove that $A_p(q)$ is negative. Since $x_q=t$, we have
\begin{equation*}
\begin{split}
&A_p(q) =
r^2_p
\left(t^2 +
y_q^2 - 2 t x_p -2 y_q y_p\right)
+
 \left(z_q   - \dfrac{x_p y_q - t y_p}{2}  \right  )^2
 -2 z_p\left(z_q   - \dfrac{x_p y_q - t y_p}{2}  \right  )
 \\
 &\leq 
 r^2_p
\left(t^2 +y_q^2 - 2 t x_p +2 |y_q y_p|\right)
+
 \left(|z_q|   +\dfrac{|x_p y_q| + t |y_p|}{2}  \right  )^2
 + 2|z_p| \left(|z_q |+  \dfrac{|x_p y_q |+ t |y_p |}{2}\right).
 \end{split}
 \end{equation*}
Note that all terms in the last inequality are positive except   $-2 t x_p$, since both $t$ and $x_p$ are positive. 

We now use the conditions $|y_q| < t \tan \theta$, $|z_q|<b $,
$x_p>a$, $|y_p| < x_p \tan \theta$, $|z_p|<b $, $b<1$ and $\tan \theta < 1 $, since $\theta< \pi/4$, to get
\begin{equation*}
\begin{split}
A_p(q)&\leq
r^2_p
\left(
t^2 + t^2 \tan^2 \theta - 2 t x_p +2   x_p    t \tan^2 \theta\right)
+
 \left(b   + t x_p  \tan \theta   \right  )^2+ 2b^2 + 2b t x_p  \tan \theta \\
&\leq 
 - 2 t x_p r^2_p+
r^2_p
\left(
t^2 + t^2 \tan^2 \theta +2   x_p    t\tan^2 \theta \right)
+
 \left(1   + x_p t \tan \theta  \right  )^2
 +2   \left(1+  x_p t    \right  ).
 \end{split}
 \end{equation*}
  We consider now separately the case $t=1$ and $t=x_p$.
  
 For $t=1$, we  bound using Lemma~\ref{lemma:PP} 
\begin{equation*}
\begin{split}
A_p(q)&\leq 
 - 2  x_p r^2_p+r^2_p\left(1 + \tan^2 \theta +2   x_p    \tan^2 \theta \right)
+
 \left(1   + x_p \tan \theta  \right  )^2
 +2   \left(1+  x_p     \right  )\\	 
 &\leq
 - 2  c_1^2  x^3_p+
c_2^2 x^2_p
\left(
1 +  \tan^2 \theta +2   x_p    \tan^2 \theta \right)
+
 \left(1   + x_p  \tan \theta  \right  )^2
 +2   \left(1+  x_p    \right  )\\	 
 &\leq  
 -  2 \left(c_1^2  -
c_2^2    \tan^2 \theta\right) \, x^3_p +
2c_2^2 x_p^2 
+
 \left(1   + x_p   \right  )^2
 +2   \left(1+  x_p    \right  )\,.
\end{split}
 \end{equation*}
Hence $A_p(q)
\leq -  2\left( c_1^2  -
c_2^2    \tan^2 \theta\right)\,  x^3_p +o(x_p^3)$  as $x_p$ goes to infinity. Thus, choosing $\theta$ small enough so that $ c_1^2  -
c_2^2    \tan^2 \theta >0$, we get that $A_p(q)\leq 0$ provided $x_p$ is large enough.

 For $t=x_p$, we use once again Lemma~\ref{lemma:PP}  and get 
 \begin{equation*}
 \begin{split}
A_p(q)
&\leq 
 - 2  r_p^2  x^2_p+
r^2_p
\left( x_p^2 + 3 x_p^2 \tan^2 \theta \right)
+
 \left(1   + x_p ^2 \tan \theta  \right  )^2
 +2   \left(1+  x_p ^2    \right  )\\	 
 &\leq 
 -   c_1^2  x^4_p+
3 c_2^2 x^4_p \tan^2 \theta 
+
 \left(1   + x_p ^2 \tan \theta  \right  )^2
 +2   \left(1+  x_p ^2    \right  )\\	 
 &\leq 
- \left(c_1^2  -
3c_2^2  
   \tan^2 \theta 
-
\tan^2 \theta \right)  x_p ^4 + 1   + 2 x_p ^2 
 +2   \left(1+  x_p ^2    \right). 
\end{split}
\end{equation*}
Hence $A_p(q)\leq -\left(c_1^2  - 3c_2^2  \tan^2 \theta - \tan^2 \theta \right)  x_p ^4 +o(x_p^4)$  as $x_p$ goes to infinity. Thus, choosing $\theta$ small enough so that $c_1^2- 3c_2^2  \tan^2 \theta - \tan^2 \theta >0$, we get that $A_p(q)\leq 0$ provided $x_p$ is large enough.

All together we have showed that one can find $\theta_0 \in (0,\pi/4)$, depending only on $\alpha$, and for all $\theta \in (0,\theta_0(\alpha))$, some $a_0(\theta) \geq 1$, such that for $a>a_0(\theta)$ and $b<1$ and for all $p\in \PP(a, b, \theta)$, we have $$\RR(1, b,\theta) \subset B_\alpha (p,r_p) \quad \text{ and }\quad  \RR(x_p, b,\theta)  \subset B_\alpha (p,r_p).$$
Since $B_\alpha (p,r_p)$ is Euclidean convex by Proposition \ref{prop:ball_convex}, we conclude the proof noting that $\RR(t, b,\theta)  $, for $ t\in [1,x_p]$, is in the Euclidean convex hull of 
$\RR(1, b,\theta) $ and $ \RR(x_p, b,\theta) $.
\end{proof}

\noindent \textit{\textbf{Proof of Lemma~\ref{lemma:x:axis0}.}}
Let $\theta_0\in (0,\pi/4)$ be given by Lemma~\ref{lemma:x:axis}. Let $\th \in (0,\th_0)$ and let $a_0(\theta) \geq 1$ be given by Lemma~\ref{lemma:x:axis}. Let $a>a_0(\theta)$ and $b\in (0,1)$. Let $p\in\HH$ and $q\in\HH$ be such that $p\notin  B(q, r_q)$ and $q\notin  B(p, r_p)$. Let us assume with no loss of generality that $x_q\leq x_p$. Then, if both $p$ and $q$ were in $\PP(a,b,\th)$, by Lemma~\ref{lemma:x:axis} we would have $q\in  \RR(x_q,b,\th)\subset B (p,r_p)$ since
$ x_q\in [1,x_p]$. But this would contradict the assumptions. \qed

\subsection{Proof of Lemma~\ref{lemma:z:axis0}}

\begin{lemma}\label{lemma:TT}
Let $a\geq 1$ and $b>0$. Then  for all $p\in \TT(a,b)$, we have
$$r_p^2 \leq \dfrac{b^2+\sqrt{b^4+4\alpha^2}}{2\alpha^2}\, |z_p|. $$
\end{lemma}

\begin{proof} 
Let $p\in \TT(a,b)$ (see \eqref{e:defT} for the definition of $\TT(a,b)$).  Since $1\leq a < |z_p|$ and $\rho_p< b$, we have (recall \eqref{formula:norm1})
\begin{equation*}
\begin{split}
r_p^2 &\leq
\dfrac{ |z_p|  \rho_p^2  +\sqrt{ z_p^2 \rho_p^4+4\alpha^2 z_p^2}}{2\alpha^2} \\
&=
\dfrac{ \rho_p^2  +\sqrt{  \rho_p^4+4\alpha^2 } }{2\alpha^2} \,|z_p|  \\
&\leq
\dfrac{ b^2  +\sqrt{  b^4+4\alpha^2 } }{2\alpha^2}\, |z_p|  \, .
\end{split}
\end{equation*}
\end{proof}

For $t\in \R$ and $b>0$, we set (see  Figure \ref{fig3}(b)) 
$$\DD(t, b):=\{p\in \HH;\;  z_q=t  , \;\rho_p<b\}.$$

\begin{lemma}\label{lemma:z:axis}
There exists $a_1 \geq 1$ and $b_1 \in (0,1)$, depending only on $\al$, such that for all $a>a_1$ and all $b\in (0,b_1)$, we have 
$$\DD(t,b) \subset B(p,r_p)$$
for all $p\in \TT(a, b)$ and all $t\in [z_p,-1]$.
\end{lemma}

\begin{proof} Take $a\geq 1 > b$, $p\in \TT(a, b)$, $t<0$ and consider $q\in \DD(t,b)$. By Lemma~\ref{A(p,q)}, showing that $q\in B (p,r_p) $ is equivalent to prove that $A_p(q)$ is negative. Since $z_q=t$, we have
\begin{equation*}
\begin{split}
A_p(q)&= r^2_p
\left(x_q^2 +y_q^2 - 2 x_q x_p - 2 y_q y_p\right)
+
 \left(t   - \dfrac{x_p y_q - x_q y_p}{2}  \right  )^2
 -2 z_p\left(t   - \dfrac{x_p y_q - x_q y_p}{2}  \right  )
 \\
 &\leq 
 r^2_p
\left(x_q^2 +y_q^2 + 2 |x_q x_p| + 2|y_q y_p|\right)
+
 \left(|t|   +\dfrac{|x_p y_q| + |x_q y_p|}{2}  \right)^2\\
 &\phantom{fggh} -2 t z_p   +  |z_p|\left(|x_p y_q| +| x_q y_p|  \right).
 \end{split}
 \end{equation*}
Note that all terms in the last inequality are positive except   $-2 t z_p$, assuming both $t$ and $z_p$ negative. 
We bound using Lemma~\ref{lemma:TT} and using that the absolute value of each of the first two components of $p$ and $q$ is smaller than $b$,
\begin{equation*}
\begin{split}
A_p(q)&\leq 6 \,  \dfrac{b^2+\sqrt{b^4+4\alpha^2}}{2\alpha^2}
\,b^2|z_p|
+
 \left(|t|   +b^2  \right  )^2
 -2 t z_p   + 2b^2 |z_p|\\
&\leq  
 -  z_p 
+
 \left(|t|   +1  \right  )^2
 -2 t z_p 
 \, ,   
  \end{split}
  \end{equation*}
   where in the last inequality we assumed that $b$ is small enough, $b<b_1$ for some $b_1$ which depends only on $\alpha$.
   
   We consider now separately the case $t=-1$ and $t=z_p$. For $t=-1$, we need $  z_p 
+
 4
 \leq 0$ which is true as soon as $z_p \leq -4$. For     $t=z_p$, we need 
$ -  z_p 
+
 \left(-z_p    +1  \right  )^2
 -2  z_p^2 = - z_p^2 - 3 z_p +1 \leq 0 
$ 
which is true as soon as $|z_p|$ is large enough.

 All together we showed that one can find $a_1\geq 1$ and $b_1\in (0,1)$, depending only on $\alpha$, such that, for all $a>a_1$ and $b\in (0,b_1)$ and all $p\in \TT(a,b)$,  we have 
 $$\DD(-1,b) \subset B (p,r_p) \qquad 
 \text{ and }\quad  \DD(z_p,b) \subset B (p,r_p).$$

 Recall that the set 
 $B (p,r_p)$ is Euclidean convex by Proposition \ref{prop:ball_convex}.
Therefore we conclude the proof since  $\DD(t,b) $, for $ t\in [z_p,-1]$, is in the Euclidean convex hull of 
$\DD(-1,b)$ and $ \DD(z_p,b)$.
\end{proof}

\noindent \textbf{\textit{Proof of Lemma~\ref{lemma:z:axis0}.}}
Let $a_1\geq 1$ and  $b_1\in (0,1)$ be given by Lemma~\ref{lemma:z:axis}. Let $a>a_1$ and $b\in (0,b_1)$. Let $p\in\HH$ and $q \in \HH$ be such that
$p\notin  B(q, r_q)$ and $q\notin  B(p, r_p)$. Assume with no loss of generality that $z_p\leq z_q$. Then, if both $p$ and $q$ were in $\TT(a,b)$, by Lemma~\ref{lemma:z:axis} we would have
$q\in  \DD(z_q,b)\subset B(p,r_p)$ since
$ z_q\in [z_p,-1]$. But this would contradict the assumptions. \qed

\section{Proof of Lemma~\ref{lemma:comparisonincone}}
\label{section:prooflemma:comparisonincone}

This section is devoted to the proof of Lemma~\ref{lemma:comparisonincone}. We first fix some notations. For $z\in \R$, we set $p_z := (0,0,z)$. 

For $\theta \in (0,\pi / 2)$, $p\in\HH$ and $z\in\R$, let $\mathcal{C}(z,\pi(p),\theta)$ denote the two dimensional Euclidean half cone in $\HH \simeq \R^3$ contained in the plane $\{q\in\HH; \,z_q=z\}$ with vertex $p_z$, axis the half line starting at $p_z$ and passing through $(x_p,y_p,z)$ and aperture $2\theta$.
See Figure \ref{fig4}(a).

For $\theta \in (0,\pi / 2)$, $p\in\HH$ and $z\in\R$, let $\mathcal{Q}(z,\pi(p),\theta)$ denote the two dimensional Euclidean equilateral quadrilateral contained in the plane $\{q\in\HH; \,z_q=z\}$ with vertices  $p_z$, $p^+_\theta:=(x_p-y_p \tan \theta, y_p+x_p\tan \theta,z)$, $p^-_\theta:=(x_p+y_p \tan \theta, y_p-x_p\tan \theta,z)$ and $\check{p}_z := (2x_p,2y_p,z)$. Note that it is the Euclidean convex hull in $\HH \simeq \R^3$ of these four points.
See Figure \ref{fig4}(b).

   \begin{figure}
\centering
\subfigure[The  cone $\mathcal C (\theta, \pi(p), z)$ containing the quadrilateral  $\mathcal Q (\theta, \pi(p), z)$.] % caption for subfigure a
{
    \label{fig:C_thetha_p_z}
 \includegraphics[height=5cm]{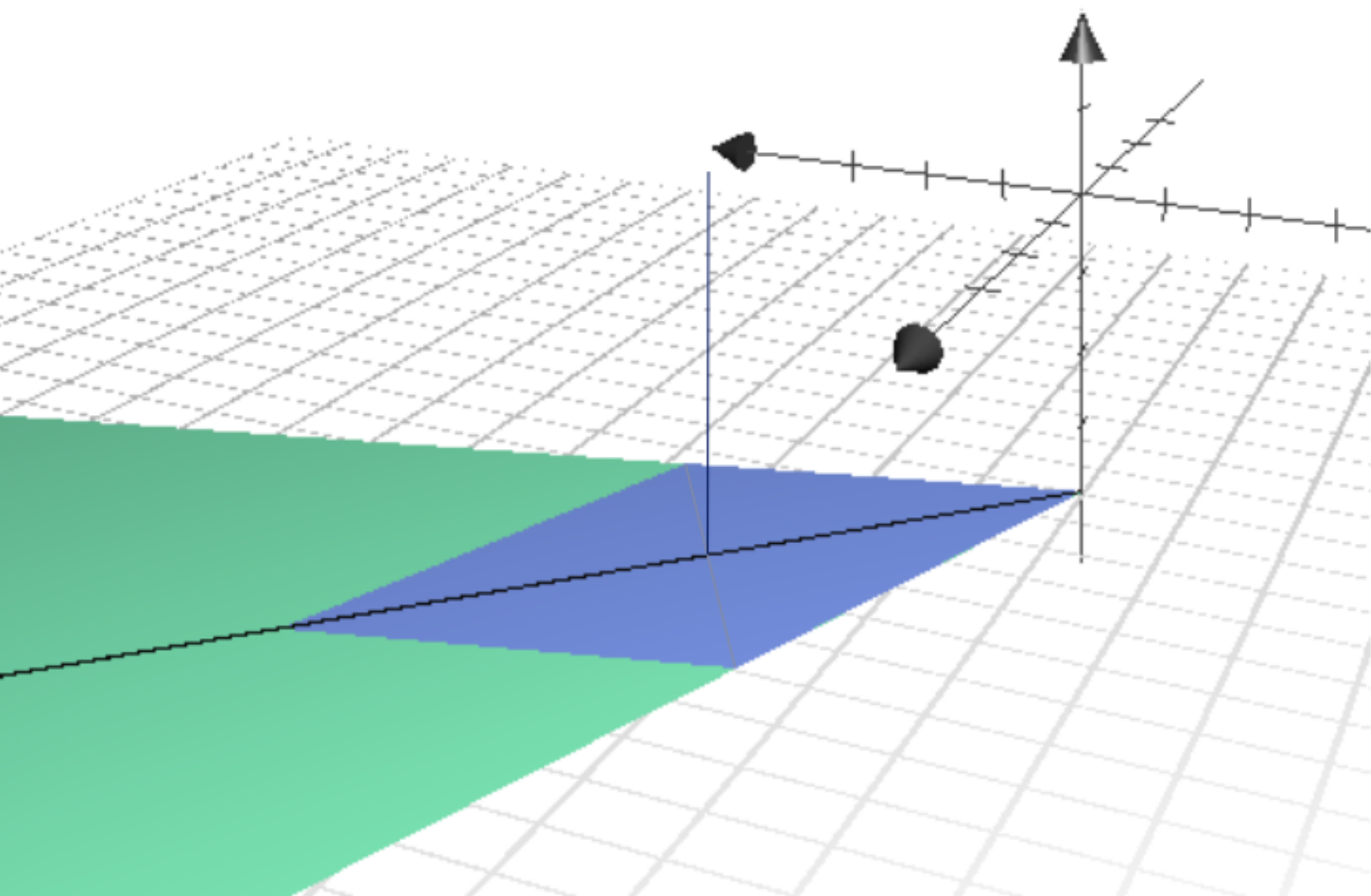}
%     \put(-120,115){$0$}
%        \put(-115,110){$.$}
             \put(-37,66){$(0,0,z)$}
        \put(-47,63.5){$.$}
             \put(-115,117){$p$}
        \put(-106.56,114){$.$}
                 \put(-94,46){$\theta$}
}
% \hspace{1cm}  
\subfigure[The quadrilateral  $\mathcal Q (\theta, \pi(p), z)$.] % caption for subfigure b
{
    \label{fig:Q}
\includegraphics[height=5.5cm]{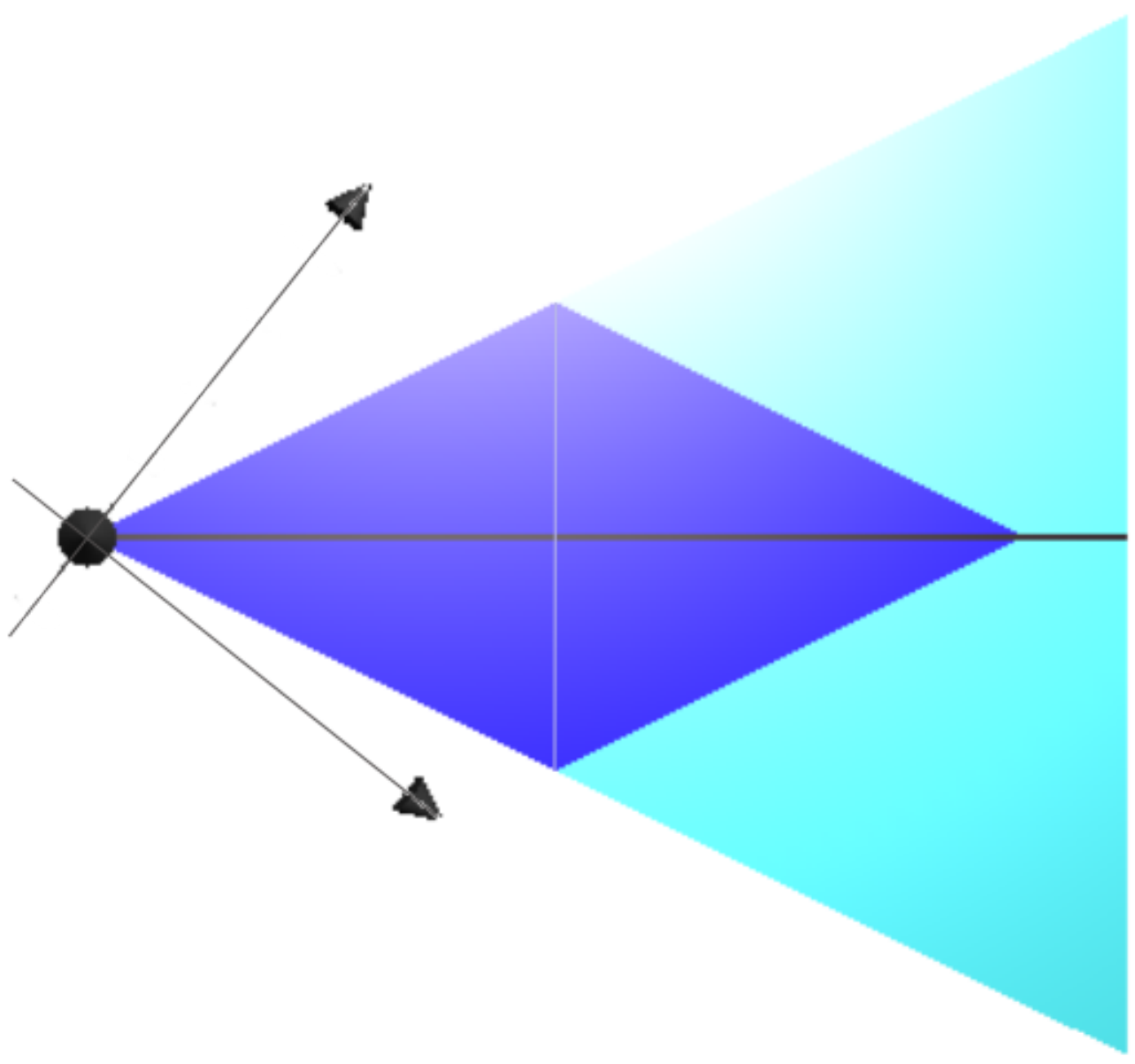}
             \put(-166,90){$(0,0,z)$}
                         \put(-15,86){$\hat p_z$} \put(-19,77){$.$}
                              
                               \put(-95,126){$p^+_\theta$}  \put(-86.5,111){$.$}
                                              \put(-104,84){$(x_p,y_p,z)$}          \put(-86.5,77){$.$}
                               \put(-95,29){$p^-_\theta$}  \put(-86.5,42.5){$.$}
}
\caption{The surfaces $\mathcal C (\theta, \pi(p), z)$ and $\mathcal Q (\theta, \pi(p), z)$.}
\label{fig4}
%\label{special segments} % caption for the whole figure
\end{figure}
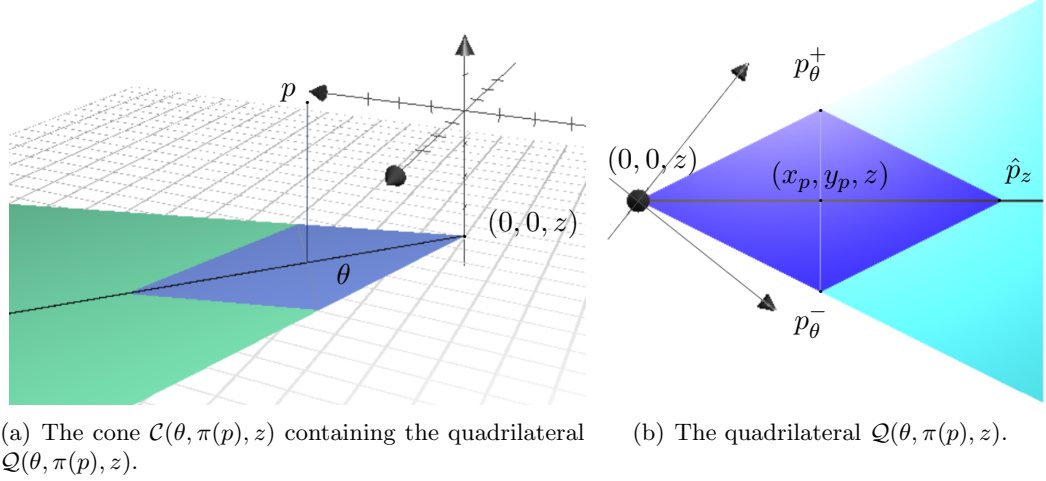

Recall \eqref{e:defCtheta} for the definition of $\mathcal{C}(\th)$. Note that $q\in \mathcal{C}(\th)$ if and only if $(x_q,y_q,0) \in \mathcal{C}(0, \pi((1,0,0)),\theta)$.

We have the following properties,
\begin{equation} \label{e:prop1}
p\in \mathcal{C}(\theta) \text{ and } q\in \mathcal{C}(\theta) \Rightarrow q\in \mathcal{C}(z_q,\pi(p),2\theta)
\end{equation}
and
\begin{equation} \label{e:prop2} 
\mathcal{Q}(z,\pi(p),\theta) \subset \mathcal{C}(z,\pi(p),\theta).
\end{equation}

For $\theta \in (0,\pi / 4)$, we have 
\begin{equation} \label{e:prop4}
\mathcal{C}(z,\pi(p),\theta) \cap \{q\in\HH;\, \rho_q \cos \theta \leq \rho_p\} \subset \mathcal{Q}(z,\pi(p),\theta).
\end{equation}

This follows from elementary geometry noting that the angle between the half lines starting at $p^+_\theta$ and passing through $p_z$ and $\check{p}_z$ respectively is larger than $\pi/2$.

\begin{lemma}\label{lemma:sev1}
There exists $\theta_2\in (0,\pi /2)$, which depends only on $\al$, such that 
\begin{equation*}
\mathcal{Q}(z,\pi(p),\theta) \subset B(p,r_p)
\end{equation*}
for all $0<\theta \leq \theta_2$, all $p\in \HH\setminus\{0\}$ and all $z\in \R$ such that $\abs{z-z_p}\leq \abs{z_p}$.
\end{lemma}

\begin{proof} Recalling Proposition \ref{prop:ball_convex}, we only need to prove that the vertices $p_z$, $p^+_\theta$, $p^-_\theta$ and $\check{p}_z$ of $\mathcal{Q}(z,\pi(p),\theta)$ belong to $B(p,r_p)$.

We have $\abs{z-z_p}\leq \abs{z_p}$ and, recalling \eqref{e:norm0'} and \eqref{e:rp}, $$\frac{\rho_p^2}{r_p^2} + \frac{|z_p|^2}{r_p^4} = \alpha^2$$
hence $$\frac{\rho_p^2}{r_p^2} + \frac{|z-z_p|^2}{r_p^4} \leq \frac{\rho_p^2}{r_p^2} + \frac{|z_p|^2}{r_p^4} = \alpha^2$$ 
that is, recalling \eqref{e:dalpha1}, $p_z = (0,0,z)\in B(p,r_p)$.

Similarly we have 
\begin{equation*}
\frac{(2x_p-x_p)^2}{r_p^2} + \frac{(2y_p-y_p)^2}{r_p^2} + \frac{|z-z_p|^2}{r_p^4} =\frac{\rho_p^2}{r_p^2} + \frac{|z-z_p|^2}{r_p^4} \leq \alpha^2
\end{equation*}
hence $\check{p}_z =(2x_p,2y_p, z) \in B(p,r_p)$.

Next, let us prove that $p^+_\theta=(x_p-y_p \tan \theta, y_p+x_p\tan \theta,z)  \in B(p,r_p)$. Set $$\Delta := \frac{ \left(z-z_p- \displaystyle \frac{\rho_p^2\tan \theta}{2}\right)^2}{r_p^4} + \frac{\rho_p^2 \tan^2\theta}{r_p^2} .$$
We need to prove that $\Delta \leq \alpha^2$. We have
\begin{equation*}
\begin{split}
\Delta &=  \frac{(z-z_p)^2}{r_p^4} + \frac{\rho_p^4\tan^2 \theta}{4r_p^4} -  \frac{\rho_p^2 (z-z_p) \tan \theta}{r_p^4}+\frac{\rho_p^2 \tan^2\theta}{r_p^2}  \\
&\leq \frac{z_p^2}{r_p^4} + \frac{\rho_p^4\tan^2 \theta}{4r_p^4} + \frac{\rho_p^2 |z-z_p| \tan \theta}{r_p^4}+\frac{\rho_p^2 \tan^2\theta}{r_p^2}\\
&\leq \alpha^2 - \frac{\rho_p^2}{r_p^2} + \frac{\rho_p^2}{r_p^2} \, \left( \frac{\alpha^2 \tan^2\theta}{4} +\alpha \tan\theta + \tan^2 \theta\right)
\end{split}
\end{equation*}
where the last inequality follows from the fact that
 $$\frac{\rho_p^2}{r_p^2} + \frac{z_p^2}{r_p^4} = \alpha^2 $$
 which implies in particular that $$\frac{\rho_p^2}{r_p^2} \leq \alpha^2 \,\text{ and }\, \frac{|z-z_p|}{r_p^2}  \leq \alpha.$$
Hence we get that 
\begin{equation*}
\Delta \leq \alpha^2 - \frac{\rho_p^2}{r_p^2} \, \left(  1 - (1+\frac{\alpha^2}{4})\tan^2 \theta - \alpha \tan\theta \right).
\end{equation*}
Choosing $\theta_2 \in (0,\pi/2)$ small enough so that $$1 - (1+\frac{\alpha^2}{4})\tan^2 \theta - \alpha \tan\theta  \geq 0$$ for all $0<\theta \leq \theta_2$, we get the conclusion.

The fact that $p^-_\theta  \in B(p,r_p)$ is proved in a similar way.
\end{proof}

\noindent \textbf{\textit{Proof of Lemma~\ref{lemma:comparisonincone}}.} Let $\th_1 = \min (\th_2/2, \pi/8)$ where $\th_2$ is given by Lemma~\ref{lemma:sev1}. Let $\th \in (0,\th_1)$ and let $p\in\HH$ and $q\in\HH$ satisfying \eqref{e:c1}, \eqref{e:c2}, \eqref{e:c3} and \eqref{e:c4}.
 
Let us first prove \eqref{e:c5}. Assume by contradiction that $2 \, z_p \leq z_q \leq 0$. Then $\abs{z_q - z_p} \leq \abs{z_p}$. Hence $\mathcal{Q}(z_q,\pi(p),2\theta) \subset B(p,r_p)$ according to Lemma~\ref{lemma:sev1}. On the other hand, it follows from \eqref{e:c3}, \eqref{e:prop1}, \eqref{e:c2}, \eqref{e:prop4} that $q\in \mathcal{Q}(z_q,\pi(p),2\theta) $ and hence $q \in B(p,r_p)$ which contradicts \eqref{e:c4}.

Thus we have $z_q<2 \, z_p \leq z_p \leq 0$ and thus $\abs{z_p - z_q} \leq \abs{z_q}$. It follows from \eqref{e:c3}, \eqref{e:prop1} and \eqref{e:c4} that $p\in \mathcal{C}(z_p,\pi(q),2\theta) \setminus B(q,r_q)$. Finally we get from Lemma~\ref{lemma:sev1} that $p\in \mathcal{C}(z_p,\pi(q),2\theta) \setminus \mathcal{Q}(z_p,\pi(q),2\theta)$ and then \eqref{e:c6} follows from \eqref{e:prop4}.\qed

\section{Two criteria for distances for which BCP does not hold} \label{section:distancesw/obcp}

In this section we prove two criteria which imply the non-validity of BCP. This shows that in some sense our example of homogeneous distance $d_\alpha$ for which BCP holds is sharp. Roughly speaking the first criterion applies to homogeneous distances  whose unit sphere centered at the origin has either inward cone-like singularities in the Euclidean sense at the poles (i.e., at the intersection of the sphere with the $z$-axis) or is flat at the poles with $0$ curvature in the Euclidean sense. The second one applies to homogeneous distances whose unit sphere at the origin has outward cone-like singularities in the Euclidean sense at the poles. Note that the unit sphere centered at the origin of our distance $d_\alpha$ is smooth with positive curvature in the Euclidean sense.

\subsection{Distances with ingoing corners or second-order flat at the poles} \label{subsection:nobcp1}
Let $d$ be a homogeneous distance on $\HH$ and let $B$ denote the closed unit ball centered at the origin in $(\HH,d)$.

In this subsection we shall most of the time identify $\HH$ with $\R^3$ equipped with its usual differential structure.

For $p\in \HH$, $\vec{v} \in\R^3$, $\vec{v}\not=(0,0,0)$, and $\alpha \in (0,\pi/2)$, let  $\cone(p,\vec{v},\alpha)$ denote the Euclidean half-cone in $\HH$, identified with $\R^3$, with vertex $p$, axis $p+\R^+\vec{v}$ and opening $2\alpha$. 

We say that $\vec{v}\in\R^3$, $\vec{v}\not=(0,0,0)$, \textit{points out of $B$ at $p\in \partial B$} if there exists an open neighbourhood $U$ of $p$ and some $\alpha \in (0,\pi/2)$ such that 
\[ B \cap \cone(p,\vec{v},\alpha) \cap U = \{p\}.\]

Let $\tau_p$ denote the left translation defined by $\tau_p(q) := p\cdot q$. We  consider it as an affine map from $\HH$, identified with $\R^3$, to $\R^3$ whose differential, in the usual Euclidean sense in $\R^3$ is thus a constant linear map and  will be denoted by $(\tau_p)_*$. Let $\hat\pi$ be defined by $\hat\pi(x,y,z) := (x,y,0)$.

For $\vec{v} \in \R^3$, $\vec{v}\not=(0,0,0)$,  and $\epsilon>0$, let $\Omega(\vec{v})$ denote the set of points $q\in \partial B$ such that $(\tau_{q^{-1}})_*(\vec{v})$ points out of $B$ at $q^{-1}$ and let $\Omega_\eps(\vec{v})$ denote the set of points $q\in \Omega(\vec{v})$ such that $\hat \pi(q)\in\R^+ \vec{w}$ for some $\vec{w} \in \im(\hat \pi)$ such that $\|\vec{w} - \vec{v}\|_{\R^3} \leq \eps$ (here $\|\cdot\|_{\R^3}$ denotes the Euclidean norm in $\R^3$).

\begin{theorem} \label{thm:nobcpingoingcorners}
Assume that there exists $\vec{v} \in \im(\hat\pi)$, $\vec{v}\not=(0,0,0)$, and $\overline\eps>0$ such that $\Omega_\eps(\vec{v})\not = \emptyset$ for all $0< \eps \leq \overline\eps$. Then BCP does not hold in $(\HH,d)$.
\end{theorem}

\begin{proof} We first construct a sequence of points $(q_n)_{n\geq 0}$ in $\partial B$ such that $q_n \in \Omega (\vec{v})$ for all $n\geq 0$ and $(\tau_{q_k^{-1}})_*(\hat\pi(q_n))$ points out of $B$ at $q_k^{-1}$ for all $n\geq 1$ and all $0\leq k \leq n-1$. 

Note that if $q\in \Omega(\vec{v})$ then there exists $\eps(q)>0$ such that $(\tau_{q^{-1}})_*(\vec{v} + \vec{\eps})$ points out of $B$ at $q^{-1}$ for all $\vec{\eps} \in\mathbb{R}^3$ such that $\|\vec{\eps} \|_{\R^3} \leq \eps(q)$ (note that the set of vectors that points out of $B$ at some point $p\in \partial B$ is open).

Let us start choosing some  $q_0\in \Omega (\vec{v})$. By induction assume that $q_0,\dots,q_n$ have already been chosen. Let $\eps = \min (\eps(q_0),\dots,\eps(q_n),\overline \eps)$ where each $\eps(q_k)$ is associated to $q_k \in \Omega(\vec{v})$ as above. Then we choose $q_{n+1}\in \Omega_\eps(\vec{v})$. We have $\hat\pi(q_{n+1})= \lambda (\vec{v} + \vec{\eps})$ for some $\lambda >0$ and some $\vec{\eps} \in \im(\hat\pi)$ such that $\|\vec{\eps} \|_{\R^3} \leq \eps$. Hence, by choice of $\eps$ and of the $q_k$'s, we have that  $(\tau_{q_k^{-1}})_*(\hat\pi(q_{n+1})) = \lambda\, (\tau_{q_k^{-1}})_*(\vec{v} + \vec{\eps})$ points out of $B$ at $q_k^{-1}$ for all $0\leq k \leq n$ as wanted.

Next, we claim that if $q\in\partial B$, $q'\in\partial B$ are such that $\hat\pi(q')\not= (0,0,0)$ and that $(\tau_{q^{-1}})_*(\hat\pi(q'))$ points out of $B$ at $q^{-1}$, then there exists $\overline \lambda>0$ such that $d(q,\delta_\lambda(q'))>1$ for all $0<\lambda\leq \overline\lambda$. Indeed the curve $\lambda \in [0,+\infty) \mapsto q^{-1}\cdot \delta_\lambda(q')$ is a smooth curve starting at $q^{-1}$ and whose tangent vector at $\lambda = 0$ is given by $(\tau_{q^{-1}})_*(\hat\pi(q'))$. Since this vector points out of $B$ at $q^{-1}$, it follows that $q^{-1}\cdot \delta_\lambda(q') \not\in B$ for all $\lambda>0$ small enough and hence $d(q,\delta_\lambda(q')) = d(0,q^{-1}\cdot \delta_\lambda(q'))>1$ as wanted. 

Then it follows that for all $n\geq 1$, one can find $\lambda_n>0$ such that for all $0< \lambda \leq \lambda_n$ and all $0\leq k<n$, one has 
\[ d(q_k,\delta_\lambda(q_n))>1. \]

Then we set $r_0=1$ and by induction it follows that we can construct a decreasing sequence $(r_n)_{n\geq 0}$ so that 
\[d(q_k,\delta_\frac{r_n}{r_k}(q_n))>1
 \]
for all $n\geq 1$ and all $0\leq k<n$. For $n\geq 0$, we set $p_n = \delta_{r_n}(q_n)$. By construction we have
\[
 d(p_k,p_n) > \max(r_k,r_n)
\]
for all $k\geq 0$ and $n\geq 0$ such that $k\not= n$. It follows that $\{B_d(p_n,r_n);\; n\in J\}$ is a family of Besicovitch balls for any finite set $J\subset \N$ and hence BCP does not hold.
\end{proof}

Let us give some examples of homogeneous distances for which the criterion given in Theorem~\ref{thm:nobcpingoingcorners} applies.

A first class of examples is given by rotationally invariant homogeneous distances $d$ that satisfy that there exists $p\in\partial B$ such that $(x_p,y_p)\not=(0,0)$ and such that $$z_p =  \max \{z>0;\, (x,y,z)\in\partial B \text{ for some } (x,y)\in\R^2 \}.$$
By rotationally invariant distances, we mean distances for which rotations $\operatorname{R}_\theta$, $\theta\in\R$, are isometries (see \eqref{e:rotations} for the definition of $\operatorname{R}_\theta$).

Indeed, consider $\vec{v}=(1,0,0)$ and, for $\varepsilon>0$, set $$\lambda=\left(\frac{x_p^2+y_p^2}{1+\varepsilon^2}\right)^{1/2}.$$
Then consider $q=(\lambda , \lambda \varepsilon,-z_p)$. By rotational and left invariance (which implies in particular that $d(0,q) = d(0,q^{-1})$ for all $q\in\HH$), one has $q\in\partial B$. On the other hand, since $\{(x,y,z)\in\HH;\, z>z_p\} \cap B =\emptyset$, any vector with a positive third coordinate points out of $B$ at $q^{-1}$. In particular $(\tau_{q^{-1}})_*(\vec{v})=(1,0,\lambda\varepsilon/2)$ points out of $B$ at $q^{-1}$. Hence $q\in \Omega_\eps(\vec{v})$.

This class of examples includes the so-called box-distance $d_\infty$ defined by $d_\infty(p,q) := \|p^{-1}\cdot q\|_\infty$ with 
\begin{equation} \label{e:box-dist}
\|p\|_\infty := \max((x_p^2+y_p^2)^{1/2}, 2\,|z_p|^{1/2})
\end{equation}
for which the fact that BCP does not hold was not known. It also includes the Carnot-Carath\'eodory distance and hence this gives a new proof of the non-validity of BCP for this distance. See \cite{Rigot} for a previous and different proof.  

Other examples of homogeneous distances $d$ for which the criterion given in Theorem~\ref{thm:nobcpingoingcorners} applies can be obtained in the following way. Assume that $B$, respectively $\partial B$, can be described as $\{q\in\HH;\; f(q)\leq 0\}$, respectively $ \{q\in\HH;\; f(q)=0\}$, for some $C^1$ real valued function $f$ in a neighbourhood of a point $p\in \partial B$. Then the outward normal to $\partial B$ at some point $q\in\partial B$ is given in a neighbourhood of $p$ by $\nabla f(q)$ (here it is still understood that we identify $\HH$ with $\R^3$ and $\nabla$ denotes the usual gradient in $\R^3$). Then Theorem~\ref{thm:nobcpingoingcorners} applies if one can find a vector $\vec{v} \in \im(\hat \pi)$, $\vec{v}\not=(0,0,0)$, such that for all $\varepsilon$ small enough, the following holds. There exists $q\in \partial B$ such that $\hat \pi(q)\in\R^+ \vec{w}$ for some $\vec{w} \in \im(\hat \pi)$ such that $\|\vec{w} - \vec{v}\|_{\R^3} \leq \eps$ and such that $q^{-1}$ lies in a neighbourhood of $p$ and 
$$\langle \nabla f (q^{-1}), (\tau_{q^{-1}})_*(\vec{v})\rangle >0$$
where $\langle \cdot,\cdot \rangle $ denotes the usual scalar product in $\R^3$.

A particular example is given when $B$, respectively $\partial B$, can be described near the north pole (intersection of $\partial B$ with the positive $z$-axis) as the subgraph $\{(x,y,z)\in\HH;\; z\leq \varphi(x,y)\}$, respectively the graph $\{(x,y,z)\in\HH;\; z= \varphi(x,y)\}$, of a $C^2$ function $\varphi$ whose first and second order partial derivatives vanish at the origin. Indeed, in that case one can choose for example $\vec{v} = (1,0,0)$ and for a fixed $\eps>0$, one looks for some $q\in \Omega_\eps(\vec{v})$ of the form $q=(\lambda, \lambda \eps, -\varphi(-\lambda,-\lambda\eps))$ for some $\lambda>0$. Then $q^{-1} = (-\lambda, -\lambda \eps, \varphi(-\lambda,-\lambda\eps))\in \partial B$ lies near the north pole for $\lambda>0$ small and we have
\begin{equation*}
\langle \nabla f (q^{-1}), (\tau_{q^{-1}})_*(\vec{v})\rangle = -\partial_x\varphi(-\lambda,-\lambda\eps) + \frac{1}{2} \lambda \eps\, 
\end{equation*}
that is equivalent to $\lambda \eps/2>0$ when $\lambda>0$ is small enough. Hence $\Omega_\eps(\vec{v})\not= \emptyset$.

This argument applies to the Cygan-Kor\'anyi distance $d_{g,2}$, and more generally to $d_{g,\alpha}$ for all values of $\alpha>0$ such that $d_{g,\alpha}$ defines a distance, thus in particular for all values of $\alpha\leq 2$. Recall from \eqref{e:n-korany-norm} that $d_{g,\alpha}(p,q) := \|p^{-1}\cdot q\|_{g,\alpha}$ where $$\|p\|_{g,\alpha}:= \left((x_p^2+y_p^2)^2+4\alpha^2\, z_p^2\right)^{1/4} $$
and that $d_{g,2}$ is the Cygan-Kor\'anyi distance. Hence Theorem~\ref{thm:nobcpingoingcorners} gives in particular a new geometric proof of the fact that BCP does not hold for the Cygan-Kor\'anyi distance on $\HH$, see \cite{KoranyiReimann} and \cite{SawyerWheeden} for previous analytic proofs.

\subsection{Distances with outgoing corners at the poles}

Let $d$ be a homogeneous distance on $\HH$ and let $B$ denote the closed unit ball centered at the origin in $(\HH,d)$. Set  $S^+:=\partial B \cap \{p\in \HH ; \, z_p>0\}$.

\begin{theorem} \label{thm:nobcpoutgoingcorners}
Assume that there exists two sequences of points $p_n^+\in S^+$ and $p_n^-\in S^+$ and some $a>0$ and $\overline x>0$ such that 
\begin{gather*}
\label{P1} p_n^-=(x_n^-, 0 , z_n^-), \quad p_n^+=(x_n^+, 0 , z_n^+),\\
\label{P2} x_n^-<0<x_n^+ , \\
\label{P3} \lim_{n\rightarrow 0} x_n^+ - x_n^- = 0,\\
\label{P4} z_n^- > z_n^+ >0, \\
\label{PP0} z_n^+- z_n^- < -a \,(x_n^+ - x_n^-) ,
\\
\label{P5} \{p\in \HH; \, x_n^+ \leq x_p \leq \overline x,\, y_p=0 , \,z_p>z_n^+\} \subset \HH \setminus B.
\end{gather*}
Then BCP does not hold in $(\HH,d)$.
\end{theorem}

The geometric meaning of the above assumptions is the following.
In some vertical plane (here we take the $xz$-plane for simplicity)
one can find two sequences of points $p_n^+$ and $p_n^-$, each one of them on a different side of the $z$-axis.
Such points are on the unit sphere centered at the origin and are converging to the north pole.
The slope between   $p_n^-$ and $p_n^+$  is assumed to be bounded away from zero.
We further assume that 
at the north pole the intersection of the sphere and the $xz$-plane can be written both as graph $x=x(z)$ and $z=z(x)$. 
See Figure \ref{out:cone}.

\begin{figure}
%\centering
%  \begin{picture}(0,-10)
%    \put(0,0)
{
    {\includegraphics[height=4cm]{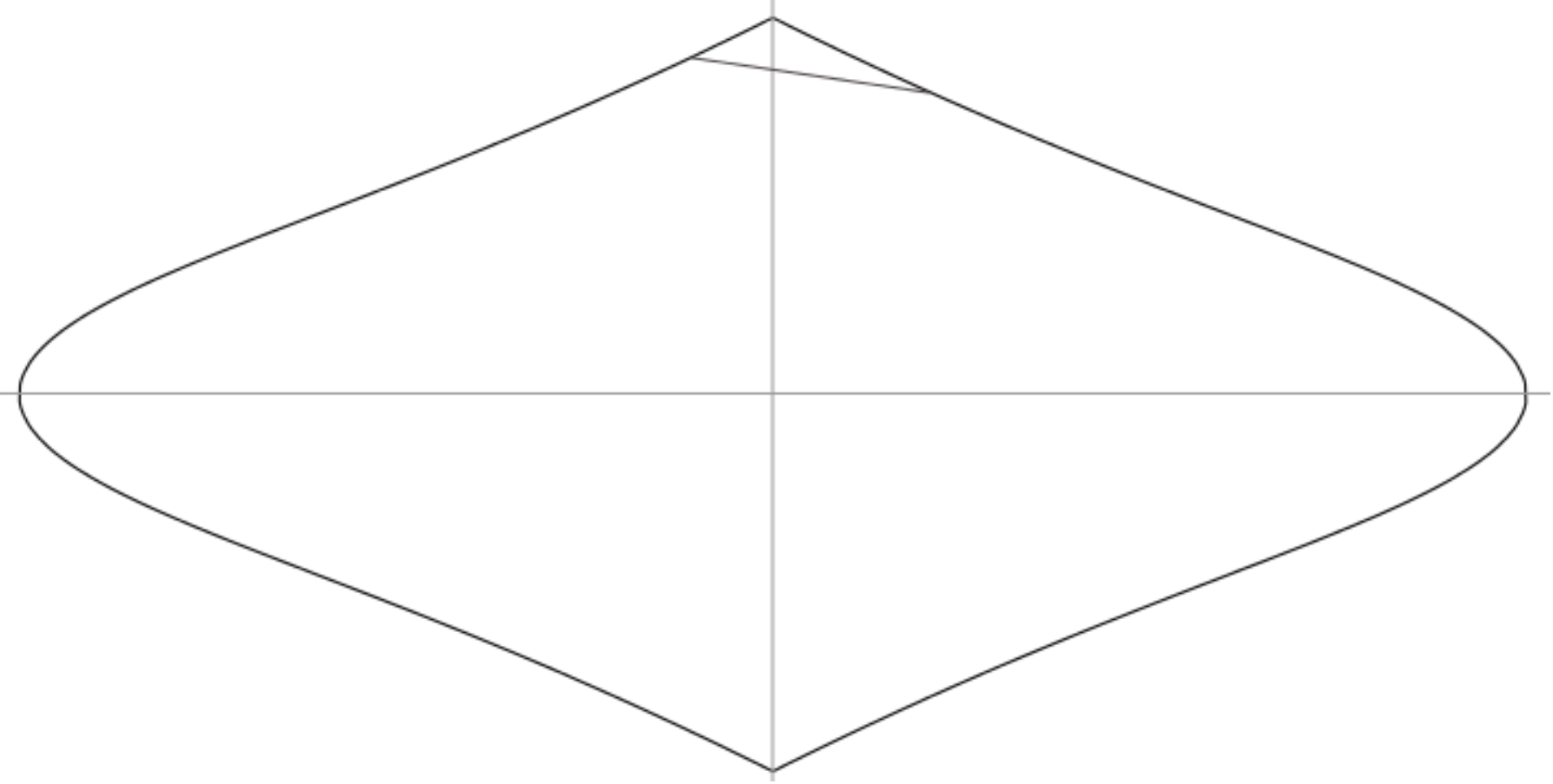}}} 
  \put(-86,102){$p_n^+$}
    \put(-142,109.4){$p_n^-$}
%    \end{picture}
    \caption{
    Intersection of the  $xz$-plane  and the unit sphere at the origin
     of the distance 
     $d_{\kappa,\alpha}$ when $\kappa=1$ and $\alpha=2$.}
	\label{out:cone}
\end{figure}

Theorem~\ref{thm:nobcpoutgoingcorners} applies in particular if
 the intersection of   $B$  with the $xz$-plane
can be described near the north pole as $ \{p\in\HH;\, -\varepsilon<x_p<\varepsilon,\;\, y_p=0,\; 0<z_p\leq f(x_p)\}$ for some function $f$ of class $C^1$ on $(-\varepsilon,\varepsilon)\setminus \{0\}$ such that $f'(0^-)$ and $f'(0^+)$ exist and are finite with $f'(0^+)<0$. This is for instance the case of the following distances built from the Cygan-Kor\'anyi distance, and more generally from the distances $d_{g,\alpha}$, and given by $d_{\kappa,\alpha}(p,q) := \|p^{-1}\cdot q\|_{\kappa,\alpha}$ with $$\|p\|_{\kappa,\alpha} := \kappa\, \rho(p) + \|p\|_{g,\alpha} $$
for some $\kappa>0$. See \eqref{e:n-korany-norm} for the definition of $\rho(\cdot)$ and $\|\cdot\|_{g,\alpha} $.
Figure \ref{out:cone} is exactly the intersection of the  $xz$-plane  and the unit sphere at the origin when $\kappa=1$ and $\alpha=2$.

Note that it follows in particular that the $l^1$-sum of the pseudo-distance $d_\rho$ with the distance $d_{g,\alpha}$ does not satisfy BCP in contrast with their $l^2$-sum which is a multiple of the distance $d_\alpha$.

\medskip

\textit{\textbf{Proof of Theorem~\ref{thm:nobcpoutgoingcorners}.}}
By induction, we construct a sequence of points $q_k=(x_k,0,z_k)$ such that
\begin{equation*}  
z_{k+1}<z_k<0<x_{k+1}<x_k \qquad \text{and} \qquad r_{k+1}>r_k
\end{equation*}
for all $k\in\N$, where $r_k = d(0,q_k)$, and such that 
\begin{equation*} 
q_l \not \in B_d(q_{k+1},r_{k+1}) 
\end{equation*}
for all $k\in\N$ and all $0\leq l \leq k$.

Then, we will have $d(q_l,q_k) > \max(r_l,r_k)$ for all $l\in\N$ and $k\in\N$ such that $l\not= k$, so that $\{B_d(q_k,r_k);\, k\in J\}$ is a family of Besicovitch balls for any finite set $J\subset \N$. Hence BCP does not hold.

We start from a point $q_0=(x_0,0,z_0)$ with $z_0<0<x_0$. Next assume that $q_0,\cdots,q_k$ have been constructed and choose $n$ large enough so that 
\begin{gather}
r_k < \frac{x_k}{x_{n}^+ - x_{n}^-}\, , \label{e:constructionqk1}\\
-a < \frac{(x_{n}^+ - x_{n}^-) }{x_k^2} \, z_k <0 \label{e:constructionqk2}\\
\intertext{and}
x_0  \leq \frac{x_k}{x_{n}^+ - x_{n}^-}\, \overline{x} \, . \label{e:constructionqk3}
\end{gather}

We set 
\begin{equation} \label{e:constructionqk}
r_{k+1}:= \frac{x_k}{x_{n}^+ - x_{n}^-} \qquad \text{and} \qquad q_{k+1} := \delta_{r_{k+1}}(p_n^-)^{-1}.
\end{equation}
Note that $d(0,q_{k+1})=r_{k+1}$ since $p_n^- \in \partial B$. We have $r_{k+1} > r_k$ by choice of $n$ (see \eqref{e:constructionqk1}). We also have $$x_{k+1} = -r_{k+1} x_n^- = \frac{-x_{n}^-}{x_{n}^+ - x_{n}^-} \, x_k < x_k\,.$$
Hence it remains to check that $z_{k+1} < z_{k}$ and that $q_l\not \in B(q_{k+1},r_{k+1})$ for $0\leq l \leq k$.

Using dilation, left translation and the assumption $\{p\in \HH;\, x_n^+ \leq x_p \leq \overline x,\, y_p=0 , \,z_p>z_n^+\} \subset \HH \setminus B$, it follows that 
\begin{equation*}
\{p\in\HH;\, x_k \leq x_p \leq r_{k+1} \overline{x} - r_{k+1} x_n^-,\, y_p = 0,\, z_p > z_{k+1} + r_{k+1}^2 z_n^+\} \subset \HH \setminus B(q_{k+1},r_{k+1}).
\end{equation*}
Hence, taking into account the fact that $z_k<\cdots<z_0$ and that $x_k<\cdots<x_0$, to prove that $z_{k+1} < z_{k}$ and that $q_l\not \in B(q_{k+1},r_{k+1})$ for $0\leq l \leq k$, we only need to check that $x_0 \leq r_{k+1} \overline{x} - r_{k+1} x_n^-$, which follows from \eqref{e:constructionqk3}, and that $z_k > z_{k+1} + r_{k+1}^2 z_n^+$. Using the fact that $z_n^+- z_n^- < -a \,(x_n^+ - x_n^-)$, \eqref{e:constructionqk2} and \eqref{e:constructionqk}, we have
\begin{equation*}
\begin{split}
z_{k+1} + r_{k+1}^2 z_n^+ &= r_{k+1}^2 \, (z_n^+ - z_n^-) \\
&< -a \, (x_n^+ - x_n^-) \, r_{k+1}^2 \\
&< \frac{(x_{n}^+ - x_{n}^-)^2\, z_k}{x_k^2} \cdot \frac{x_k^2}{(x_{n}^+ - x_{n}^-)^2} = z_k
\end{split}
\end{equation*}
which gives the conclusion.
\qed

\section{Generalization to any Heisenberg group $\HH^n$} \label{section:Hn}

The case of $\HH^n$ for $n\geq 1$ arbitrary can be easily handled similarly to the case of $\HH$ adopting the following convention. For $p\in\HH^n$, we set $p=(x_p,y_p,z_p)$ where $x_p\in\R$, $y_p\in\R^{2n-1}$ and $z_p\in\R$. Note that this is different from the more standard presentation adopted in the introduction (Section~\ref{section:introduction}). To avoid any confusion, the explicit correspondance between theses two conventions is the following. If $x=(x_1,\cdots,x_n)\in\R^n$, $y=(y_1,\cdots,y_n)\in\R^n$ and $z\in\R$ denote the exponential and homogeneous coordinates of $p\in \HH^n$ as in \eqref{e:grouplaw}, by denoting $p=(x_p,y_p,z_p)$ with $x_p\in\R$, $y_p\in\R^{2n-1}$ and $z_p\in\R$, we mean  $x_p = x_1$, $y_p=(x_2,\cdots,x_n,y_1,\cdots,y_n)$ and $z_p=z$. It follows that $y_p^2$ should be replaced by $\|y_p\|_{\R^{2n-1}}^2$ and $|y_p|$ by $\|y_p\|_{\R^{2n-1}}$ where $\|\cdot\|_{\R^{2n-1}}$ denotes the Euclidean norm in $\R^{2n-1}$.

In particular, we get
\begin{equation*} 
\rho_p = \sqrt{x_p^2+ \|y_p\|_{\R^{2n-1}}^2}
\end{equation*}
and setting
\begin{equation*} 
\PP(a,b,\theta):=\{p\in \HH^n; \; x_p>a, \; |z_p|<b, \; \|y_p\|_{\R^{2n-1}} <x_p \tan \theta \}
\end{equation*}
and 
\begin{equation*} 
\TT(a,b):=\{p\in \HH^n ; \; z_p<-a, \;\rho_p<b\},
\end{equation*}
one can easily check that Lemma~\ref{lemma:x:axis0} and Lemma~\ref{lemma:z:axis0} hold true in $\HH^n$ with essentially the same proofs.

Lemma~\ref{lemma:comparisonincone} and its proof extend to the case of $\HH^n$ setting
\begin{equation*} 
\mathcal{C}(\theta) := \{p\in\HH^n;\,  \|y_p\|_{\R^{2n-1}}<  x_p \tan \theta \}
\end{equation*} 
and considering the analogue of the sets $\mathcal{C}(z,\pi(p),\theta)$ and $\mathcal{Q}(z,\pi(p),\theta)$ (introduced in Section~\ref{section:prooflemma:comparisonincone}) defined in the following way. 

The set $\mathcal{C}(z,\pi(p),\theta)$ is now defined as the $(2n)$-dimensional Euclidean half cone contained in the hyperplane $\{q\in\HH^n;\; z_q=z\}$ with vertex $p_z=(0,0,z)$, axis the half line starting at $p_z$ and passing through $(x_p,y_p,z)$ and aperture $2\theta$.

The set $\mathcal{Q}(z,\pi(p),\theta)$ is defined as the $(2n)$-dimensional Euclidean convex hull in the hyperplane $\{q\in\HH^n;\; z_q=z\}$ of $p_z$, $\check{p}_z = (2x_p,2y_p,z)$ and the $(2n-1)$-dimensional Euclidean ball $\{q\in\HH^n;\; z_q=z,\; \langle \pi(q)-\pi(p),\pi(p)\rangle_{\R^{2n}} = 0,\; \|\pi(q)-\pi(p)\|_{\R^{2n}} = \rho_p \tan \theta\}$. Here $\pi$ denotes the obvious analogue of the map defined in \eqref{e:projection}, $\pi:\HH^n \rightarrow \R^{2n}$, $\pi(x_p,y_p,z_p) := (x_p,y_p)$.

\section{A general construction giving bi-Lipschitz equivalent distances without BCP} \label{section:destroybcp}

This section is devoted to the proof of Theorem~\ref{thm:destroybcp}. The construction is inspired by the construction given by the first-named author in Theorem 1.6 of \cite{ledonne} where it is proved that there exist translation-invariant distances on $\R$ that are bi-Lipschitz equivalent to the Euclidean distance but that do not satisfy BCP.

\begin{proof}[\textbf{Proof of Theorem~\ref{thm:destroybcp}}] Let $(M,d)$ be a metric space. Assume that $\overline{x}$ is an accumulation point in $(M,d)$ and let  $(x_n)_{n\geq 1}$ be a sequence of distinct points in $M$ such that $x_n\not= \overline{x}$ for all $n\geq 1$ and such that $\lim_{n\rightarrow +\infty} d(x_n,\overline{x})=0$. Set $$\rho_n := \displaystyle \frac{n}{n+1} \, d(x_n,\overline{x})\,.$$

Up to a subsequence, one can assume with no loss of generality that the sequence $(\rho_n)_{n\geq 1}$ is decreasing.

Let $0<c<1$ be fixed and $n_0 \in \N$ be fixed large enough so that 
\begin{equation}\label{eq:c}
c\, (n_0+1) < n_0.
\end{equation}

Set
$$\theta(x,y) := \begin{cases} \rho_n \,\,\,\,\qquad \text{ if } \{x,y\}=\{x_n,\overline{x}\} \text{ for some } n\geq n_0\\
d(x,y) \quad \text{ otherwise }
\end{cases}
$$
and $$ \overline{d}(x,y) := \inf \sum_{i=0}^{N-1} \theta(a_i,a_{i+1}) $$
where the infimum is taken over all $N\in \N^*$ and all chains of points $a_0= x,\dots, a_N=y$.  

\medskip

Then $\overline{d}$ is a distance on $M$ such that $c \, d \leq \overline{d} \leq d$. This follows from Lemma~\ref{lem:equivalentdist} and Lemma~\ref{lem:distance} below.

\medskip

Next, we will prove that $\overline{x}$ is an isolated point of $B_{\overline{d}}(x_n,\rho_n)$ for all $n\geq n_0$. More precisely, by definition of $\overline{d}$, we have, for all $n\geq n_0$,
\begin{equation*}
\overline{d}(x_n,\overline{x}) \leq \theta(x_n,\overline{x}) = \rho_n\,,
\end{equation*}
hence $\overline{x} \in B_{\overline{d}}(x_n,\rho_n)$ for all $n\geq n_0$. On the other hand, we will prove in Lemma~\ref{lem:isolatedpoint} that 
\begin{equation} \label{e:isolatedpoint}
B_{\overline{d}}(x_n,\rho_n) \cap B_d(\overline{x},\frac{\rho_n}{n(n+1)}) = \{\overline{x}\}\,
\end{equation}
for all $n\geq n_0$.

Then let us extract a subsequence $(x_{n_k})_{k\geq 0}$ starting at $x_{n_0}$ in such a way that 
\begin{equation*} 
d(\overline{x},x_{n_k}) < \frac{\rho_{n_j}}{n_j(n_j+1)}
\end{equation*}
for all $k\geq 1$ and all $j\in\{0,\dots,k-1\}$. It follows from \eqref{e:isolatedpoint} that 
\begin{equation*}
\overline{d}(x_{n_k},x_{n_j}) > \rho_{n_j} = \max\{\rho_{n_j},\rho_{n_k}\}
\end{equation*}
for all $k\geq 1$ and all $j\in\{0,\dots,k-1\}$ (remember that the sequence $(\rho_h)_{h\geq 1}$ is assumed to be decreasing).

Then $\{B_{\overline{d}}(x_{n_k}, \rho_{n_k});\, k\in J\}$  is a family of Besicovitch balls for any finite set $J\subset \N$ which implies that w-BCP, and hence BCP, do not hold in $(M,\overline d)$. 
\end{proof}

\medskip

\begin{lem} \label{lem:equivalentdist}
We have $c \, d \leq \overline{d} \leq d$.
\end{lem}

\begin{proof}
By definition of $\theta$, one has $\theta(x,y)\leq d(x,y)$ for all $x\in M$ and $y\in M$. It follows that 
\begin{equation*}
\overline{d}(x,y) \leq \inf (\sum_{i=0}^{N-1} d(a_i,a_{i+1}); a_0= x,\dots, a_N=y) = d(x,y).
\end{equation*}
Note that since $d$ is a distance, one indeed has $$d(x,y)=\inf (\sum_{i=0}^{N-1} d(a_i,a_{i+1}); a_0= x,\dots, a_N=y)$$ which follows from one side from the triangle inequality and for the other side from  the fact that one can consider $N=1$, $a_0= x$ and $a_1=y$, so that $d(x,y)\geq \inf (\sum_{i=0}^{N-1} d(a_i,a_{i+1}); a_0= x,\dots, a_N=y)$.

\medskip

On the other hand, 
since $s\mapsto s/(s+1)$ is increasing, it follows from the definition of $\theta(x,y)$ and from \eqref{eq:c} that one has
\begin{equation} \label{e:thetad}
\theta(x,y) \geq  \frac{n_0}{n_0+1} \, d(x,y)\geq c \, d(x,y)
\end{equation}
for all $x\in M$ and $y\in M$. Hence 
\begin{equation*}
\overline{d}(x,y) \geq c \inf (\sum_{i=0}^{N-1} d(a_i,a_{i+1}); a_0= x,\dots, a_N=y) = c\, d(x,y)\,.
\end{equation*}
\end{proof}

\begin{lem} \label{lem:distance}
We have that $\overline{d}$ is a distance on $M$.
\end{lem}

\begin{proof}
We get from Lemma~\ref{lem:equivalentdist} that if $\overline{d}(x,y)=0$ then $d(x,y)=0$ and hence $x=y$. Since $\theta(x,y) = \theta(y,x)$, one has $\overline{d}(x,y)=\overline{d}(y,x)$. To prove the triangle inequality, let us consider $x$, $y$ and $z$ in $M$ and two arbitrary chains of points $a_0=x,\dots,a_N=z$, $b_0=z,\dots,b_{N'}=y$. Since $a_0=x,\dots,a_N=z = b_0, \dots, b_{N'}=y$ is a chain of points from $x$ to $y$, one has
\begin{equation*}
\overline{d}(x,y) \leq \sum_{i=0}^{N-1} \theta(a_i,a_{i+1}) + \sum_{i=0}^{N'-1} \theta(b_i,b_{i+1})
\end{equation*}
and hence
\begin{equation*}
\overline{d}(x,y) \leq \overline{d}(x,z) + \overline{d}(z,y).
\end{equation*}
\end{proof}

\begin{lem} \label{lem:isolatedpoint}
Let $n\geq n_0$. Assume that $0< d(\overline{x},y) < \displaystyle\frac{\rho_n}{n(n+1)}$. Then  $\overline{d}(x_n,y) > \rho_n$.
\end{lem}

\begin{proof}
By contradiction, assume that $0< d(\overline{x},y) < \displaystyle\frac{\rho_n}{n(n+1)}$ for some $n\geq n_0$ and   $\overline{d}(x_n,y) \leq \rho_n$. Let $\eps>0$ and $a_0=x_n,\dots,a_N=y$ be such that 
\begin{equation}  \label{e:0}
\sum_{i=0}^{N-1} \theta(a_i,a_{i+1}) \leq \rho_n + \eps. 
\end{equation}

First, we claim that $\{a_i,a_{i+1}\}\not=\{x_n,\overline{x}\}$ for all $i\in\{0,\dots,N-1\}$ provided $\eps$ is small enough. Indeed, otherwise, with no loss of generality, we would have $a_0=x_n$ and $a_1=\overline{x}$, and hence
\begin{equation*}  
\sum_{i=0}^{N-1} \theta(a_i,a_{i+1}) = \theta(x_n,\overline{x}) + \sum_{i=1}^{N-1} \theta(a_i,a_{i+1}) = \rho_n + \sum_{i=1}^{N-1} \theta(a_i,a_{i+1}) \leq \rho_n +\eps
\end{equation*}
which implies that
$$\sum_{i=1}^{N-1} \theta(a_i,a_{i+1}) \leq \eps\,.$$
On the other hand, \eqref{e:thetad} together with the triangle inequality would give
\begin{equation*}
c \, d(\overline{x},y) \leq c \sum_{i=1}^{N-1} d(a_i,a_{i+1}) \leq \sum_{i=1}^{N-1} \theta(a_i,a_{i+1})\leq \eps
\end{equation*}
which is impossible as soon as $\eps < c \, d(\overline{x},y)$.

Next, we claim that
\begin{equation} \label{e:1}
\theta(a_i,a_{i+1})\geq \frac{n+1}{n+2} \, d(a_i,a_{i+1})
\end{equation}
for all $i\in\{0,\dots,N-1\}$.

Indeed, first, if $\{a_i,a_{i+1}\} = \{\overline{x},x_m\}$ for some $m\geq n_0$, then we must have $m>n$. Otherwise, since $(\rho_h)_{h\geq 1}$ is decreasing, we would have $\rho_m\geq \rho_{n-1}$. Hence we would get
$$\rho_{n-1}\leq \rho_m = \theta(a_i,a_{i+1}) \leq \sum_{j=0}^{N-1} \theta(a_j,a_{j+1}) \leq \rho_n +\eps$$
which is impossible as soon as $\eps <\rho_{n-1} - \rho_n$.

Next, if $\{a_i,a_{i+1}\} = \{\overline{x},x_m\}$ for some $m>n$, then, by definition of $\theta$ and remembering that $s\mapsto s/(s+1)$ is increasing, we have 
$$\theta(a_i,a_{i+1}) = \rho_m = \frac{m}{m+1}\, d(a_i,a_{i+1}) \geq \frac{n+1}{n+2} \,d(a_i,a_{i+1})$$
which gives \eqref{e:1}.

Finally, if $\{a_i,a_{i+1}\} \not= \{\overline{x},x_m\}$ for all $m\geq n_0$, then it follows from the definition of $\theta$ that
$$\theta(a_i,a_{i+1}) = d(a_i,a_{i+1}) \geq \frac{n+1}{n+2} \,d(a_i,a_{i+1})$$
which also gives \eqref{e:1}.

Now, it follows from \eqref{e:0} and \eqref{e:1} that 
\begin{equation*}
%\begin{split}
\rho_n + \eps \geq \sum_{i=1}^{N-1} \theta(a_i,a_{i+1}) \geq \frac{n+1}{n+2} \sum_{i=1}^{N-1} d(a_i,a_{i+1}) \geq \frac{n+1}{n+2} \, d(x_n,y) 
%\end{split}
\end{equation*}
for all $\eps$ small enough. Letting $\eps\downarrow 0$, we get that 
\begin{equation*}
\begin{split}
\rho_n \geq \frac{n+1}{n+2}\, d(x_n,y) \geq \frac{n+1}{n+2}\, \left(d(x_n,\overline{x}) - d(\overline{x},y)\right)
\geq \frac{n+1}{n+2} \,\left(\frac{n+1}{n}\rho_n - d(\overline{x},y)\right) 
\end{split}
\end{equation*}
and hence $$d(\overline{x},y) \geq \frac{\rho_n}{n(n+1)}$$ which contradicts the assumptions and concludes the proof.
\end{proof}

\medskip

\noindent \textbf{Acknowledgement.} The authors are very grateful to Jeremy Tyson for useful discussions and in particular for pointing out the link between the distances $d_\alpha$ and the distances of negative type considered by Lee and Naor. The second author would like to thank for its hospitality the Department of Mathematics and Statistics of the University of Jyv\"askyl\"a where part of this work was done.

%%%

\end{document}